\documentclass[11pt,reqno,oneside]{amsart}
\usepackage{amscd}
\usepackage{amsmath}
\usepackage{latexsym}
\usepackage{amsfonts}
\usepackage{amssymb}
\usepackage{amsthm}
\usepackage{graphicx}
\usepackage{verbatim}
\usepackage{mathrsfs}
\usepackage{enumerate}
\usepackage{hyperref}
\usepackage{fullpage}
\usepackage{color}

\theoremstyle{plain}
\newcommand{\norm}[1]{\left\|#1\right\|}
\newtheorem{Theo}{Theorem}[section]
\newtheorem{lem}[Theo]{Lemma}

\newtheorem{prop}[Theo]{Proposition}

\theoremstyle{plain}
\theoremstyle{definition}
\newtheorem{defi}[Theo]{Definition}

\theoremstyle{remark}
\newtheorem{Rema}[Theo]{Remark}
\newtheorem*{rema*}{Remark}

\newcommand{\RR}{\mathbb{R}}

\allowdisplaybreaks
\begin{document}
\title{Global well-posedness for axisymmetric MHD system with only vertical  viscosity}

\author[Q. Jiu, H. Yu and X. Zheng]{Quansen Jiu$^{1}$, Huan Yu$^{2}$ and  Xiaoxin Zheng$^{3}$}
\address{$^1$ School of Mathematical Sciences, Capital Normal University, Beijing, 100048, P.R.
China}

\email{jiuqs@mail.cnu.edu.cn}

\address{$^2$  School of Mathematical Sciences, Capital Normal University, Beijing, 100048, P.R.
China}

\email{yuhuandreamer@163.com}

\address{$^3$ School of Mathematics and Systems Science, Beihang University, Beijing 100191, P.R. China}

\email{xiaoxinzheng@buaa.edu.cn}
\date{\today}

\begin{abstract}
In this paper, we  are  concerned with  the global well-posedness of a tri-dimensional MHD system with only vertical viscosity in velocity equation for the large axisymmetric initial data. By  making good use of the axisymmetric structure of flow and the maximal smoothing effect of  vertical diffusion, we show that $\displaystyle\sup_{2\leq p<\infty}\int_0^t\frac{\|\partial_{z}u(\tau)\|_{L^p}^{2}}{p^{3/4}}\,\mathrm{d}\tau<\infty$.  With this regularity for the vertical first derivative  of velocity vector field, we further establish  losing estimates for the anisotropy tri-dimensional MHD system to get the high regularity of $(u,b)$,  which guarantees that $\int_0^t\|\nabla u(\tau)\|_{L^\infty}\,\mathrm{d}\tau<\infty$. This together with the classical commutator estimate entails the global regularity of a smooth solution.
\end{abstract}

\smallskip
\maketitle
\section{Introduction}
\setcounter{section}{1}\setcounter{equation}{0}
The magneto-hydrodynamics (MHD) equations govern the dynamics of the velocity and the magnetic field in electrically conducting fluids such as plasmas and reflect the basic physics conservation laws. It has been at the center of numerous analytical, experimental and numerical investigations. The Cauchy problem of the tri-dimensional incompressible MHD system has the following form
\begin{equation}\label{equ}
\left\{\begin{array}{ll}
(\partial_{t}+u\cdot\nabla)u-\nu_{x}\partial^{2}_{xx}u-\nu_{y}\partial^{2}_{yy}u-\nu_{z}\partial^{2}_{zz} u+\nabla p=(b\cdot\nabla) b,\quad(t,\mathrm{\mathbf{x}})\in\RR^+\times\RR^3,\\
(\partial_{t}+u\cdot\nabla)b-\eta_{x}\partial^{2}_{xx}b-\eta_{y}\partial^{2}_{yy}b-\eta_{z}\partial^{2}_{zz} b=(b\cdot\nabla) u,\\
\text{div}\,u=\text{div}\,b=0,\\
(u,b)|_{t=0}=(u_0,b_0),
\end{array}\right.
\end{equation}
where  $u=u(\mathrm{\mathbf{x}},t)$ denotes the velocity of the fluid, $b=b(\mathrm{\mathbf{x}},t)$ stands for the magnetic field and the scalar function $p=p(\mathrm{\mathbf{x}},t)$ is pressure.  The  parameters $\nu_{x}$, $\nu_{y}$, $\nu_{z}$, $\eta_{x}$, $\eta_{y}$, $\eta_{z}$ are nonnegative constants. In addition, the initial data $u_0$ and $b_0$ satisfy $\text{div}\,u_0=\text{div}\,b_0=0$ and $\mathrm{\mathbf{x}}=(x,y,z)$.

In the bi-dimensional case,  the constants $\nu_{z}$  and $\eta_{z}$ in problem \eqref{equ} become zero.  For this case, if all parameters $\nu_{x}$, $\nu_{y}$, $\eta_{x}$, $\eta_{y}$ are positive, some results concerning on the global well-posedness for  sufficiently smooth initial data (see for example \cite{DJ,ST}) were established in terms of the $L^2$-energy estimate. When the four parameters are zero, it reduces to an ideal MHD system. The global regularity of this system is still a challenging open problem. So, it has been a hot research topic to examine the intermediate cases where some of the four parameters are positive in the past few  years. Recently, Cao and Wu in \cite{CW} showed that smooth solutions are global for system  \eqref{equ} with $\nu_{x}>0$, $\nu_{y}=0$, $\eta_{x}=0$, $\eta_{y}>0$ or  $\nu_{x}=0$, $\nu_{y}>0$, $\eta_{x}>0$, $\eta_{y}=0$. More progress has also been made on several other partial dissipation cases of the bi-dimensional MHD equations. For system  \eqref{equ}  with $\nu_{x}>0$, $\eta_{x}>0$ and  $\nu_{y}=\eta_{y}=0$,  Cao, Dipendra and Wu \cite{CDW} derived  that  the horizontal component of any solution admits a global (in time) bound in any Lebesgue space $L^{2r}$ with $1 < r <\infty$ and the bound grows no faster than the order of
$r\log r$ as $r$ increases.
In \cite{CW} and \cite{LZ},  system  \eqref{equ}  with $\nu_{x}=\nu_{y}=0$, $\eta_{x} =\eta_{y} > 0$ are shown to posse global $H^{1}$ weak solutions. However, the uniqueness of such weak solutions and a global $H^{2}$-bound remain unknown.
Very recently, when $\eta_{x}=\eta_{y}=0$,  $\nu_{x} =\nu_{y}> 0$ in  system  \eqref{equ}, the global well-posedness by assuming that the initial data  is close to a non-trivial
steady state was investigated in \cite{LXZ}, \cite{Ren} and \cite{TZhang}.

Nevertheless, except that the initial data have some special structures, it is still not known whether or not the tri-dimensional Navier-Stokes system (when $b=0$ in \eqref{equ}) with large initial data has a unique global smooth solution. For instance, by assuming that  the initial data is axisymmetric without swirl,  Ladyzhenskaya \cite{LA} and Ukhovskii and Yudovich \cite{UY} independently  proved that weak solutions are regular for all time (see also \cite{PO}). Inspired by \cite{PO}, \cite{LA} and \cite{UY}, more recent works are devoted to considering the  axisymmetric Boussinesq or MHD system without swirl component of the velocity field. The global regularity results have been obtained for the axisymmetric Boussinesq system without swirl, when the dissipation only occurs in one equation or is  present only in one direction (anisotropic dissipation)(see, e.g., \cite{1,rd,HKR,HR,MZ,mz}). While for the axisymmetric MHD system, the first result that a specific geometrical assumption allows global well-posedness was established by Lei in \cite{L}.
More precisely,  under the assumption that $u_{\theta}$, $b_{r}$ and $b_{z}$ are trivial,  he showed that there exists a unique global solution if the initial data is smooth enough. Later on,  Jiu and Liu \cite{JL} further investigated the Cauchy problem for the tri-dimensional axisymmetric MHD equations with horizontal dissipation and vertical magnetic diffusion.

In the present paper, we aim at investigating  system \eqref{equ} with $\nu_{x}=\nu_{y}=0$, $\eta_{x}=\eta_{y}=\eta_{z}=0$.  Without loss of generality, we  set $\nu_{z}=1$.   The corresponding system thus reads
\begin{equation}\label{bs}
   \left\{\begin{array}{ll}
       (\partial_{t}+u\cdot\nabla)u-\partial^{2}_{zz} u+\nabla p=(b\cdot\nabla) b,\quad(t,\mathrm{\mathbf{x}})\in\mathbb{R}^{+}\times\mathbb{R}^{3},\\
       (\partial_{t}+u\cdot\nabla)b=(b\cdot\nabla)u,\\
       \text{div}\,u=\text{div}\,b=0,\\
       (u,b)|_{t=0}=(u_{0},b_{0}).
\end{array}\right.
\end{equation}
Our main concern here is to establish a family of unique solutions of system \eqref{bs} with the form
\begin{equation*}
u(\mathrm{\mathbf{x}},t)=u_{r}(r,z,t)e_{r}+u_{z}(r,z,t)e_{z},
\end{equation*}
\begin{equation*}
b(\mathrm{\mathbf{x}},t)=b_{\theta}(r,z,t)e_{\theta}
\end{equation*}
in the cylindrical coordinate system. Here
$$e_{r}=(\frac{x}{r}, \frac{y}{r}, 0),\quad e_{\theta}=(-\frac{y}{r}, \frac{x}{r}, 0) ,\quad e_{z}=(0,0,1),\quad \text{and}\quad  r=\sqrt{x^{2}+y^{2}}.$$
 Then, we  can equivalently reformulate \eqref{bs} as
\begin{equation}\label{MHDaxi}
\left\{\begin{array}{ll}
(\partial_{t}+u\cdot\nabla) u_{r}-\partial_{zz}^{2}u_{r}=-\partial_{r}p-\frac{b^{2}_{\theta}}{r},\quad(t,\mathrm{\mathbf{x}})\in\mathbb{R}^{+}\times\mathbb{R}^{3},\\

(\partial_{t}+u\cdot\nabla) u_{z}-\partial_{zz}^{2}u_{z}=-\partial_{z}p,\\

(\partial_{t}+u\cdot\nabla) b_{\theta}=\frac{b_{\theta}u_{r}}{r},\\

\partial_{r}u_{r}+\frac{u_{r}}{r}+\partial_{z}u_{z}=0,\\
       (u_{r},u_{z},b_{\theta})|_{t=0}=(u_{0}^{r},u_{0}^{z},b_{0}^{\theta}),
\end{array}\right.
\end{equation}
in the cylindrical coordinates.

By easy computations,  we find that  the vorticity $\omega:=\nabla\times u$ can be expressed as
\begin{equation*}
\omega(\mathrm{\mathbf{x}},t)=\omega_{\theta}(r,z,t)e_{\theta},
\end{equation*}
with
$$\omega_{\theta}=\partial_{z} u_{r}-\partial_{r} u_{z}.$$
It follows from \eqref{MHDaxi} that $\omega_{\theta}$ satisfies
\begin{equation}\label{w}
\partial_{t}\omega_{\theta}+(u\cdot\nabla)\omega_{\theta}-\partial_{zz}^{2}\omega_{\theta}=
\frac{u_{r}\omega_{\theta}}{r}-\frac{\partial_{z}b^{2}_{\theta}}{r}.
\end{equation}
Now let us present the main result.
\begin{Theo}\label{the}
Let $(u_0,b_0)\in H^{s}(\RR^3)\times H^{s}(\RR^3)$ with $s>\frac{5}{2}$ be axisymmetric divergence free vector fields such that $u_{0}=u_{0}^{r}e_{r}+u_{0}^{z}e_{z}$ and $b_{0}=b_{0}^{\theta}e_{\theta}$. Then system \eqref{bs} admits a unique global-in-time axisymmetric solution $(u,b)$ satisfying
  $$u\in C\big(\RR^+;H^{s}(\RR^3)\big),\quad \partial_{z}u\in L^{2}_{\rm loc}\big(\RR^+;H^{s}(\RR^3)\big),$$
$$b\in C\big(\RR^+;H^{s}(\RR^3)\big),\quad \frac{b_{\theta}}{r}\in C\big(\RR^+;H^{s-1}(\RR^3)\big).$$
\end{Theo}
\begin{Rema}
Compared with the  model considered in \cite{L}, the diffusion only occurs in vertical direction of the velocity equation of system \eqref{bs}. This leads to the fact that  the method used in \cite{L} doesn't work for problem  \eqref{bs}. Thus we develop some new estimates and techniques to compensate for the loss of the horizontal diffusion. Also, the assumption of anisotropy is natural and  reasonable, because experiments show that in certain regimes and after
suitable rescaling, the horizontal dissipation is negligible
compared to  the vertical dissipation.
\end{Rema}
\begin{Rema}
Formally, system \eqref{bs} corresponds to the following bi-dimensional MHD equations
\begin{equation*}
\left\{\begin{array}{ll}
\partial_tu+(u\cdot\nabla)u-\partial_{yy}^2u+\nabla p=(b\cdot\nabla)b,\quad(t,\mathrm{\mathbf{x}})\in\RR^+\times\RR^2,\\
\partial_tb+(u\cdot\nabla)b=(b\cdot\nabla)u,\\
\mathrm{div}\,u=\mathrm{div}\,b=0,
\end{array}\right.
\end{equation*}
where $\mathrm{\mathbf{x}}=(x,y).$ But,  the global well-posedness of these  equations for the  large initial data is still an open problem. Cao and Wu \cite{CW} just showed the global regularity for the bi-dimensional MHD equations with mixed partial dissipation and magnetic diffusion.
\end{Rema}
Now, we briefly describe the difficulties and outline the main ingredients in our proof. Since the viscosity  occurs only in the vertical direction of the velocity equation,  behavior of the system is like that of the inviscid incompressible MHD equations.  Thus, how to establish $\int_0^t\|\nabla u(\tau)\|_{L^{\infty}}\,\mathrm{d}\tau<\infty$ for any $t>0$ is the key point to establish global solution for large initial data. To do this, we meet a big problem that it is impossible to use the $L^2$ energy estimate to control the strong coupling nonlinearities between the velocity and the magnetic fields.  This  induces us to consider the solution which enjoys the special structure. Inspired by the ideas in \cite{L,LA}, we consider the solution with the form
\begin{equation*}
u(\mathrm{\mathbf{x}},t)=u_{r}(r,z,t)e_{r}+u_{z}(r,z,t)e_{z}\quad\text{and}\quad
b(\mathrm{\mathbf{x}},t)=b_{\theta}(r,z,t)e_{\theta}.
\end{equation*}
Thanks to the special structure of this solution to system \eqref{bs},
we observe that the quantity  $\frac{b_\theta}{r}$ satisfies the homogeneous transport equation
\begin{equation}\label{b-equ}
\partial_{t}\Big(\frac{b_\theta}{r}\Big)+
(u\cdot\nabla)\Big(\frac{b_\theta}{r}\Big)=0.
\end{equation}
The incompressible condition allows us to get the maximum principle of $\frac{b_\theta}{r}$:
$$\Big\|\frac{b_{\theta}}{r}(t)\Big\|_{L^p}\leq \Big\|\frac{b_{0}}{r}\Big\|_{L^p}  \quad\text{for}\quad p\in [2,\infty].
$$
As a result, we are able to bound $\|\frac{\omega_{\theta}}{r}(t)\|_{L^{3,1}}$
and $\|b_{\theta}(t)\|_{L^{p}}$ with $p\in[2,\infty]$, where $L^{3,1}$ is a Lorentz space (see Section 2 for details).

Next, taking the standard $L^{p}$-estimate of the vorticity equation \eqref{w} yields
  \begin{equation*}
\frac1p\frac{\mathrm{d}}{\mathrm{d}t}\|\omega_\theta(t)\|_{L^p}^p+\frac{4(p-1)}{p^{2}}
\big\|\partial_{z}|\omega_\theta(t)|^{\frac{p}{2}}\big\|_{L^2}^2
=\int_{\mathbb{R}^3}  \frac{u_{r}}{r}|\omega_\theta|^{p} \,\mathrm{d}\mathrm{\mathbf{x}}+(p-1)\int_{\mathbb{R}^3}  \frac{b_{\theta}^{2}}{r}|\omega_\theta|^{p-2} \partial_{z}\omega_\theta\,\mathrm{d}\mathrm{\mathbf{x}}.
\end{equation*}
Taking advantage of the H\"{o}lder inequlity and the fact that $\big\|\frac{u_{r}}{r}\big\|_{L^{\infty}}\leq \big\|\frac{\omega_{\theta}}{r}\big\|_{L^{3,1}}$, we see that
  \begin{equation*}
\int_{\mathbb{R}^3}  \frac{u_{r}}{r}|\omega_\theta|^{p} \,\mathrm{d}\mathrm{\mathbf{x}}\leq  C\Big\|\frac{\omega_{\theta}}{r}\Big\|_{L^{3,1}}\|\omega_\theta\|_{L^p}^{p}.
  \end{equation*}
  On the other hand, by the Young inequality  and the H\"older inequality,  one has
\begin{align*}
(p-1)\int_{\mathbb{R}^3}  \frac{b_{\theta}^{2}}{r}|\omega_\theta|^{p-2} \partial_{z}\omega_\theta\,\mathrm{d}\mathrm{\mathbf{x}}\leq\frac{2(p-1)}{p^{2}}\big\|\partial_{z}|\omega_\theta|^{\frac {p}{2}}\big\|_{L^2}^{2}+Cp
\Big(\Big\|\frac{b_{\theta}}{r}\Big\|_{L^{2p}}^{4}+\|b_{\theta}\|_{L^{2p}}^{4}\Big)\|\omega_\theta\|_{L^p}^{p-2}.
\end{align*}
Collecting both estimates, it follows that
\begin{equation}\label{vorticity-estimate}
\|\omega(t)\|_{\sqrt{\mathbb{L}}}:=\sup_{2\leq p<\infty}\frac{\|\omega_{\theta}(t)\|_{L^{p}}}{{\sqrt{p}}}<
\infty.
\end{equation}
This together with the well-known fact $$\|\nabla u\|_{L^p}\leq C\frac{p^2}{p-1}\|\omega\|_{L^p}\quad\text{with}\quad p\in(1,\infty)$$ leads to $\|\nabla u\|_{L^p}\leq Cp^{\frac32}$.
Unfortunately, the function $p\sqrt{p}$ does not belong to the dual Osgood modulus of continuity (an dual Osgood modulus of continuity $\Omega(p)$ is the non-decreasing function satisfying $\int_a^\infty\frac{1}{\Omega(\tau)}\,\mathrm{d}\tau=\infty$ for some $a>0$). In other words, the growth rate of $\|\nabla u\|_{L^p}$  is  too fast to obtain higher-order estimates of $(u, b)$. To overcome this difficulty, we further exploit the space-time estimate about $\displaystyle\sup_{2\leq p<\infty}\int_0^t\frac{\norm{\partial_{z}u(\tau)}_{L^p}^{2}}{p^{3/4}}
\,\mathrm{d}\tau<\infty$ by making good use of the maximal vertical  smoothing effect and micro-local techniques.
This space-time estimate enables us to obtain the limited loss of the high regularity for $\omega_{\theta},\, b_\theta$ and $\frac{b_\theta}{r}$. Since this loss is arbitrary small, we can show that  $\int_0^t\|\nabla u(\tau)\|_{L^{\infty}}\,\mathrm{d}\tau<\infty$ for all $t\geq0$,  which we believe to be of independent interest (see Proposition~\ref{losingestimate--1} in Section \ref{Sec-3}). With this regularity for $u$,  the BKM's criterion ensures the global regularity of problem~\eqref{bs}.

The rest of the paper is organized as follows. In Section \ref{Sec-2}, we  review  Littlwood-Paley theory  and some useful lemmas and  then establish losing a priori estimates for transport equation,  which is an important ingredient in the proof of Theorem \ref{the}.  In Section \ref{Sec-3}, we  obtain a priori estimates for sufficiently smooth solutions of  system \eqref{bs} by using the procedure that we have just described in introduction. \mbox{Section \ref{Sec-4}} is devoted to the proof of Theorem \ref{the}. Finally, an appendix is devoted to several useful lemmas.

\noindent \emph{Notation}: Throughout  the paper,  the $L^{p}(\RR^{3})$-norm of a function $f$ is denoted by $\norm{f}_{L^p}$ and the $H^{s}(\RR^{3})$-norm by $\norm{f}_{H^s}$. Moreover, $L^{q}_{T}(X)=L^{q}(0,T;X)$, $1\leq q\leq\infty$   is the set of function $f(t)$ defined on
$(0, T)$ with values in a given Banach space $X(\RR^3)$ such that $\int_{0}^{T}\norm{f(t)}_{X(\RR^3)}\,\mathrm{d}t<\infty$ and we denote $X(\RR^3)$ by $X$ for simplicity. The spaces ${\mathbb{L}}^{a}$ with $a\in[0,1]$ consist of all functions $f\in L^{p}$, $2\leq p<\infty$ satisfying
\begin{equation*}
\norm{f}_{\mathbb{L}^{a}}:=\sup_{2\leq p<\infty}\frac{\norm{f}_{L^{p}}}{{p}^{a}}<\infty.
\end{equation*}
We denote ${\mathbb{L}^{\frac12}}$ by $\sqrt{{\mathbb{L}}}$ for the sake of simplicity.

\section{Preliminaries}\label{Sec-2}
\setcounter{section}{2}\setcounter{equation}{0}
This section consists of  three subsections. In the first subsection, we recall the Littlewood-Paley theory and introduce Besov spaces. In the second  subsection, we provide the Bernstein-type inequalities for fractional derivatives and some lemmas used for the proof of Theorem \ref{the}. The third  subsection is devoted to  the proof of  losing a prior estimates for the anisotropy transport-diffusion equations, which enables us to establish  the Lipschitz estimate for the velocity field.
\subsection{The Littlewood-Paley theory and Besov spaces}
Assume that $(\chi,\varphi)$ is a couple of smooth functions with  values in $[0,1]$
such that $\mathrm{supp}\,\chi\in\big\{\xi\in\mathbb{R}^{n}\big||\xi|\leq\frac{4}{3}\big\}$,
$\mathrm{supp}\,\varphi\subset\big\{\xi\in\mathbb{R}^{n}\big|\frac{3}{4}\leq|\xi|\leq\frac{8}{3}\big\}$ and
\begin{align*}
    \chi(\xi)+\sum_{j\in \mathbb{N}}\varphi(2^{-j}\xi)=1\quad {\rm for \ each\ }\xi\in \mathbb{R}^{n}.
\end{align*}
For every $u\in \mathcal{S}'(\mathbb{R}^{n})$,  the dyadic blocks can be defined by
\begin{equation*}
 \Delta_{-1}u=\chi(D)u\quad\text{and}\quad   {\Delta}_{j}u:=\varphi(2^{-j}D)u\quad {\rm for\ each\ }j\in\mathbb{N}.
\end{equation*}
We shall also use the  following low-frequency operator:
\begin{equation*}
    {S}_{j}u:=\chi(2^{-j}D)u.
\end{equation*}
From the definition of operators above, it is easy to check that
\begin{equation*}
    u=\sum_{j\geq-1}{\Delta}_{j}u\quad\text{in}\quad\mathcal{S}'(\mathbb{R}^{n}).
\end{equation*}
Moreover,  we introduce the Bony para-product decomposition to deal with the nonlinear term. For two tempered distributions $u$ and $v$, we define para-product term and remainder term as follows:
$$
T_{u}v=\displaystyle{\sum_{j}}S_{j-1}u\Delta_{j}v,  \qquad R(u,v)=\displaystyle{\sum_{|i-j|\leq 2}}\Delta_{i}u\Delta_{j}v
$$
and then we have the following Bony's decomposition
\begin{equation*}
uv=T_{u}v+T_{v}u+R(u,v).
\end{equation*}

\begin{defi}\label{def2.2}
For $s\in \mathbb{R}$, $(p,q)\in [1,+\infty]^{2}$ and $u\in \mathcal{S}'(\mathbb{R}^{n})$, we set
\begin{equation*}
    \norm{u}_{{B}^{s}_{p,q}(\mathbb{R}^{n})}:=
    \Big(\sum_{j\geq-1}2^{jsq}
    \norm{{\Delta}_{j}u}_{L^{p}(\mathbb{R}^{n})}^{q}\Big)^{\frac{1}{q}}
    \quad\text{if}\quad q<+\infty
\end{equation*}
and
\begin{equation*}
    \norm{u}_{{B}^{s}_{p,\infty}(\mathbb{R}^{n})}:=\sup_{j\geq-1}2^{js}\norm{{\Delta}_{j}u}_{L^{p}(\mathbb{R}^{n})}.
\end{equation*}
Then we define  \emph{inhomogeneous Besov spaces} as
\begin{equation*}
   {B}^{s}_{p,q}(\mathbb{R}^{n}):=\big\{u\in\mathcal{S}'(\mathbb{R}^{n})\big|\norm{u}_{{B}^{s}_{p,q}(\mathbb{R}^{n})}<+\infty\big\}.
\end{equation*}
\end{defi}

Next, we introduce the anisotropic spaces because the dissipation term only occurs in the vertical direction. To do this, we need to define the following anisotropic operator:
 $$\Delta_i^hf(\mathrm{\mathbf{x}}_h):=2^{2i}\int_{\RR^2}\varphi(\mathrm{\mathbf{x}}_h-2^i \mathrm{\mathbf{y}})f(\mathrm{\mathbf{y}})\,\mathrm{d}\mathrm{\mathbf{y}} \quad\text{and}\quad  \Delta_i^vf(z):=2^{2i}\int_{\RR}\varphi(z-2^i y)f(y)\,\mathrm{d}y$$
 for $i=0,1,2\cdots$ and $\mathrm{\mathbf{x}}_h:=(x,y)$.
 \begin{defi}\label{def-anisotropic}
For $\alpha,\beta \in \mathbb{R}$, $(p,q)\in [1,+\infty]^{2}$ and $u\in \mathcal{S}'(\mathbb{R}^{3})$, we set
\begin{equation*}
    \norm{u}_{{B}^{\alpha,\beta}_{p,q}(\mathbb{R}^{3})}:=
  \Big(\sum_{j,\,k\geq-1}2^{j\alpha q}2^{k\beta q}\big\|{\Delta}^{h}_{j}{\Delta}^{v}_{k}u\big\|_{L^{p}(\mathbb{R}^{3})}^{q}\Big)^{\frac{1}{q}}
    \quad\text{if}\quad q<+\infty
\end{equation*}
and
\begin{equation*}
    \norm{u}_{{B}^{\alpha,\beta}_{p,\infty}(\mathbb{R}^{3})}:=\sup_{j,\,k\geq-1}2^{j\alpha}2^{k\beta}\big\|{\Delta}^{h}_{j}{\Delta}^{v}_{k}u\big\|_{L^{p}(\mathbb{R}^{3})}.
\end{equation*}
Then \emph{anisotropic Besov spaces}  are  defined by
\begin{equation*}
   {B}^{\alpha,\beta}_{p,q}(\mathbb{R}^{3}):=\big\{u\in\mathcal{S}'(\mathbb{R}^{3})\big|\norm{u}_{{B}^{\alpha,\beta}_{p,q}(\mathbb{R}^{3})}<+\infty\big\}.
\end{equation*}
\end{defi}

Let us point out that the usual Sobolev spaces $H^\alpha$ and $H^{\alpha,\beta}$ coincide with  Besov spaces  $B_{2,2}^\alpha$ and $B_{2,2}^{\alpha,\beta}$, respectively.

Finally, we review Lorentz spaces and the generalized Young inequality for the convolution of two functions. Let $1<p<\infty$ and $1\leq q\leq\infty$. Then, by the classical real interpolation, one can define Lorentz spaces $L^{p,q}$ as follows:
\begin{equation*}
L^{p,q}(\mathbb{R}^n):=\big(L^{p_1}(\mathbb{R}^n),L^{p_2}(\mathbb{R}^n)\big)_{(\theta,q)},
\end{equation*}
where $1\leq p_1<p<p_2\leq\infty$ satisfy $\frac1p=\frac{\theta}{p_1}+\frac{1-\theta}{p_2}$ and $q\in[1,\infty].$

According to the definition above, we easily find that for $1<p<\infty$ and $1\leq q_1\leq q_2\leq\infty,$
\begin{equation*}
L^{p,q_2}(\mathbb{R}^n)\hookrightarrow L^{p,q_1}(\mathbb{R}^n)\quad\text{and}\quad L^{p}(\mathbb{R}^n)=L^{p,p}(\mathbb{R}^n).
\end{equation*}
\begin{lem}[\cite{ONeil}]\label{con}
Let $1<p,q,r<\infty$, $0< s_1,s_2\leq\infty,$ $\frac1p+\frac1q=\frac1r+1$, and $\frac{1}{s_1}+\frac{1}{s_2}=\frac1s.$ Then there holds
               \[\|f\ast g\|_{L^{r,s}(\mathbb{R}^n)}\leq C(p,q,s_1,s_2)\|f\|_{L^{p,s_1}(\mathbb{R}^n)}\|g\|_{L^{q,s_2}(\mathbb{R}^n)}.\]
Moreover, in the case that $r=\infty$, we have that for $\alpha\in(0,n),$
 \[\|f\ast g\|_{L^{\infty}(\mathbb{R}^n)}\leq C \|f\|_{L^{\frac{n}{\alpha},\infty}(\mathbb{R}^n)}\|g\|_{L^{\frac{n}{n-\alpha},1}(\mathbb{R}^n)}.\]
\end{lem}
\subsection{Bernstein inequalities and some useful lemmas}
Bernstein inequalities are useful tools in dealing with Fourier localized functions
and  integrability for derivatives. The following lemma
provides Bernstein-type inequalities for fractional derivatives.
\begin{lem}\label{bern}
Let $\alpha\ge0$ and  $1\le p\le q\le \infty$.
\begin{enumerate}[\rm 1)]
\item  If $f$ satisfies
$$
\mathrm{supp}\, \widehat{f} \subset \{\xi\in \mathbb{R}^n: \,\, |\xi|
\le K 2^j \},
$$
for some integer $j$ and a constant $K>0$, then there exists a constant $C_1>0$ such that
$$
\|\Lambda^{\alpha} f\|_{L^q(\mathbb{R}^n)} \le C_1\, 2^{\alpha j +
j n(\frac{1}{p}-\frac{1}{q})} \|f\|_{L^p(\mathbb{R}^n)}.
$$
\item  If $f$ satisfies
\begin{equation*}\label{spp}
\mathrm{supp}\, \widehat{f} \subset \{\xi\in \mathbb{R}^n: \,\, K_12^j
\le |\xi| \le K_2 2^j \}
\end{equation*}
for some integer $j$ and constants $0<K_1\le K_2$,  then there exist two constants $C_2>0$  and $C_3>0$ such that
$$
C_2\, 2^{\alpha j} \|f\|_{L^q(\mathbb{R}^n)} \le \|\Lambda^\alpha
f\|_{L^q(\mathbb{R}^n)} \le C_3\, 2^{\alpha j} \|f\|_{L^q(\mathbb{R}^n)}.
$$
\end{enumerate}
\end{lem}
\begin{lem}[\cite{CZ}]\label{lem-p}
For any $p\in(1,\infty)$, there exists a positive constant $C$ independent of $p$ such that
\begin{equation}
\|\nabla u\|_{L^p(\RR^n)}\leq C\frac{p^2}{p-1}\|\omega\|_{L^p(\RR^n)}.
\end{equation}
\end{lem}
 \begin{lem}[\cite{bcd, CWZ2012}]\label{heat}
Assume that $f$ solves the classical linear heat equation
\begin{equation*}
\left\{\begin{array}{ll}
\partial_t f-\Delta f=0,\quad(t, \mathrm{\mathbf{x}})\in \mathbb{R}^+\times\mathbb{R}^{n},\\
f|_{t=0}=f_{0},
\end{array}\right.
\end{equation*}
Then there  exist two positive constants $C$ and $c$ such that  for every $j\geq 0$,
 \begin{equation*}
\|\Delta_{j}f (t) \|_{L^{p}(\mathbb{R}^{n})}=\|e^{t\Delta }\Delta_{j}f_{0}\|_{L^{p}(\mathbb{R}^{n})}\leq Ce^{-ct2^{2j}}\|f_{0}\|_{L^{p}(\mathbb{R}^{n})}.
 \end{equation*}
\end{lem}

Next, we recall a useful algebraic identity and its properties.

\begin{lem}\label{point}
Let $u$ be a divergence-free axisymmetric vector-field  and $\omega=\mathrm{curl}\, u$. Then
\begin{equation}\label{eq.M-Z}
\frac{u_{r}}{r}=\partial_{z}\Delta^{-1}\Big(\frac{\omega_{\theta}}{r}\Big)
-2\frac{\partial_{r}}{r}\Delta^{-1}
\partial_{z}\Delta^{-1}\Big(\frac{\omega_{\theta}}{r}\Big).
\end{equation}
Moreover, we have
$$\norm{\partial_{z}\Big(\frac{u_{r}}{r}\Big)}_{L^p(\RR^{3})}\leq
C\norm{\frac{\omega_{\theta}}{r}}_{L^p(\RR^{3})}, ~1<p<\infty.$$
$$\norm{\partial_{zz}\Big(\frac{u_{r}}{r}\Big)}_{L^p(\RR^{3})}\leq
C\Big\|\partial_{z}\Big(\frac{\omega_{\theta}}{r}\Big)\Big\|_{L^p(\RR^{3})}, ~1<p<\infty.$$
\end{lem}
\begin{proof}
The magic algebraic identity \eqref{eq.M-Z} was established by Miao and Zheng, one can refer to \cite{mz} for the proof.
\end{proof}

The following lemma is about anisotropic Sobolev norms, which will be useful in the proof of Proposition \ref{prop-fine}.

\begin{lem}\label{prop-fra}
For any $\alpha\in[0,1)$,  there holds the following  estimates
\begin{equation}\label{pre1}
\|\Lambda_v^{\alpha}u\|_{L^2(\RR^3)}\leq C(\|u\|_{L^{2}(\RR^3)}+\|\omega\|_{\sqrt{\mathbb{L}}(\RR^3)})
\end{equation}and
\begin{equation}\label{pre2}
\|\Lambda_v^{\alpha}u\|_{L^{\infty}(\RR^3)}\leq C(\|u\|_{L^{2}(\RR^3)}+\|\omega\|_{\sqrt{\mathbb{L}}(\RR^3)}).
\end{equation}
Here and in what follows, we denote $\Lambda_h:=\sqrt{-(\partial_1^2+\partial_2^2)}$ and $\Lambda_v:=\sqrt{-\partial_z^2}$.
\end{lem}
\begin{proof}
By the interpolation inequality,  there exists a $p\in[2,\infty)$ such that
\begin{equation*}
\begin{split}
\|\Lambda_v^{\alpha}u\|_{L^2(\RR^3)}\leq \|\Lambda^{\alpha}u\|_{L^2(\RR^3)}
\leq &C\|u\|_{L^2(\RR^3)}^{\theta}\|\nabla u\|_{L^p(\RR^3)}^{1-\theta}\\ \leq &C\|u\|_{L^2(\RR^3)}^{\theta}\|\omega\|_{L^p(\RR^3)}^{1-\theta},
\end{split}
\end{equation*}
where $\theta=\frac{\frac{5}{2}-\frac{3}{p}-\alpha}{\frac{5}{2}-\frac{3}{p}}$.

Then,  the Young inequality and the definition of space $\sqrt{\mathbb{L}}$ yield the first desired estimate~\eqref{pre1} in Lemma \ref{prop-fra}.

Now we need to show the second desired estimate. By the Littlewood-Paley decomposition and Bernstein inequality, we have
\begin{align*}
\|\Lambda_v^{\alpha}u\|_{L^{\infty}(\RR^3)} & \leq\|\Delta_{-1}\Lambda_v^{\alpha}u\|_{L^\infty(\RR^3)}
+ \sum_{k=0}^\infty \|\Delta_{k}\Lambda_v^{\alpha}u\|_{L^\infty(\RR^3)} \\
& \leq C \|\Lambda_v^{\alpha}u\|_{L^2(\RR^3)} + C\, \sum_{k=0}^\infty 2^{k\alpha} \,\|\Delta_k u\|_{L^\infty(\RR^3)} \\
& \leq C \|\Lambda_v^{\alpha}u\|_{L^2(\RR^3)} + C\, \sum_{k=0}^\infty 2^{(\alpha-1 + \frac{3}{p})k} \, \|\Delta_k \omega\|_{L^p(\RR^3)}\\
& \leq C \|\Lambda_v^{\alpha}u\|_{L^2(\RR^3)} + C\,\|\omega\|_{L^p(\RR^3)}\, \sum_{k=0}^\infty 2^{(\alpha-1 + \frac{3}{p})k}.
\end{align*}
It is noted that   for any $\alpha\in[0,1)$ and some $p\in[2,\infty)$ satisfing $\alpha-1 + \frac{3}{p} <0$, $$\sum_{k=0}^\infty 2^{(\alpha-1 + \frac{3}{p})k}<\infty.$$
This combined with estimate \eqref{pre1} yields
\begin{equation*}
\|\Lambda_v^{\alpha}u\|_{L^\infty(\RR^3)}\leq C(\|u\|_{L^2(\RR^3)}+\|\omega\|_{\sqrt{\mathbb{L}}(\RR^3)}).
\end{equation*}
We complete  the proof of Lemma \ref{prop-fra}.
\end{proof}
\subsection{Losing a priori estimates for the anisotropy transport-diffusion equations}
This subsection is mainly devoted to the proof of losing a priori estimates for the following system
\begin{equation}\label{losingestimate}
   \left\{\begin{array}{ll}
       \partial_{t}\rho+(u\cdot\nabla) \rho -\partial^{2}_{zz}\rho=f+\partial_{z}g ,\quad (t,\mathrm{\mathbf{x}})\in[0,T]\times\mathbb{R}^3,\\
       \mathrm{div}\,u=0,\\
       \rho|_{t=0}=\rho_{0},
\end{array}\right.
\end{equation}
under  the condition that the vertical first derivative  of  $u$  satisfies
\begin{equation}\label{l-5}
\sup_{2\leq q<\infty}\int_{0}^{T}\frac{\norm{\partial_{z}u(t)}_{L^q}^{2}}{q^{3/4}}\,\mathrm{d}t+\int_0^T\Big\|\frac{u_r}{r}(t)\Big\|^2_{L^\infty}\,\mathrm{d}t<\infty,
\end{equation}
and the vorticity of $u$ satisfies
\begin{equation}\label{l-4}
\int_{0}^{T}\norm{\omega(t)}_{\sqrt\mathbb{L}}\,\mathrm{d}t<\infty.
\end{equation}

Now, let us begin with the statement of losing a priori estimates for the ordinary transport equation. More precisely, we have

\begin{prop}\label{losingestimate-1}
Let $\sigma\in (-1,1)$ and $p\in[2,\infty]$.  Assume that $\rho$ satisfies the following transport equation
\begin{equation}\label{losingestimate-2}
   \left\{\begin{array}{ll}
       \partial_{t}\rho+(u\cdot\nabla )\rho=f,\quad(t,\mathrm{\mathbf{x}})\in [0,T]\times\mathbb{R}^3,\\
       \mathrm{div}\,u=0,\\
       \rho|_{t=0}=\rho_{0}
\end{array}\right.
\end{equation}
with initial data  $\rho_{0}\in B_{p,\infty}^{\sigma}$ and force term $f\in L^{2}_{T}(B_{p,\infty}^{\sigma})$. Assume in addition that \eqref{l-5} and \eqref{l-4} hold.
Then, there exists a positive constant $C=C(p,\sigma,T)$ such that the following estimates hold for all small enough $\epsilon>0$:
\begin{equation*}
\sup_{0\leq t\leq T}\|\rho(t)\|^{2}_{B_{p,\infty}^{\sigma_{t}}}\leq CU(T)e^{\frac{CU^3(T)}{\epsilon^{3}}\big(\int_{0}^{T}V(t)\,\mathrm{d}t\big)^{4}}\left(\|\rho_{0}\|_{B_{p,\infty}^{\sigma}}^{2}+
\int_{0}^{T}\|f(\tau)\|_{B_{p,\infty}^{\sigma_{\tau}}}^{2}\,\mathrm{d}\tau\right),
\end{equation*}
where $V(t):=1+\norm{\omega(t)}_{\sqrt{\mathbb{L}}}$ and $$\displaystyle{U(t):=\exp{\Big(\sup_{2\leq q<\infty}\int_{0}^{t}C
\Big(1+\frac{\norm{\partial_{z}u(\tau)}_{L^q}^{2}}{q^{3/4}}\Big)\,\mathrm{d}\tau+C\int_0^t\Big\|\frac{u_r}{r}(\tau)\Big\|^2_{L^\infty}\,\mathrm{d}\tau\Big)}}.$$
In particular, we have for all small enough $\epsilon>0$:
\begin{equation*}
\|\rho\|_{L^{\infty}_{T}(B_{p,\infty}^{\sigma-\epsilon})}\leq CU^{\frac12}(T)e^{\frac{CU^3(T)}{\epsilon^{3}}\big(\int_{0}^{T}V(t)
\,\mathrm{d}t\big)^{4}}\big(\|\rho_{0}\|_{B_{p,\infty}^{\sigma}}
+\|f\|_{L^{2}_{T}(B_{p,\infty}^{\sigma})}\big).
\end{equation*}
Here and in what follows, for any $t\in [0,T]$, $$ \sigma_{t}=\sigma-\eta\int_{0}^{t}V(\tau)\,\mathrm{d}\tau, \quad \eta=\frac{\epsilon}{\int_{0}^{T}V(t)\,\mathrm{d}t}.$$
\end{prop}
\begin{proof}
Applying  $\Delta_{q}$ to the  equation satisfied by $\rho$ leads to
$$\partial_{t}\Delta_{q}\rho+(S_{q+1}u\cdot\nabla)\Delta_{q}\rho
=\Delta_{q}f+R_{q}(u,\rho),$$
with $$R_{q}(u,\rho)=(S_{q+1}u\cdot\nabla)\Delta_{q}\rho-\Delta_{q}\big((u\cdot\nabla)\rho\big).$$
Multiplying the above equation by $|\Delta_{q}\rho|^{p-2}\Delta_{q}\rho$, integrating by parts and using H\"{o}lder's inequality  yield
\begin{equation*}
\frac{1}{p}\frac{\mathrm{d}}{\mathrm{d}t}\|\Delta_{q}\rho(t)\|_{L^{p}}^{p}
\leq\|\Delta_{q}f\|_{L^{p}}\|\Delta_{q}\rho\|_{L^{p}}^{p-1}+
\|R_{q}(u,\rho)\|_{L^p}\|\Delta_{q}\rho\|_{L^p}^{p-1}
\end{equation*}
whence,
\begin{equation*}
\frac{\mathrm{d}}{\mathrm{d}t}\|\Delta_{q}\rho(t)\|_{L^p}^{2}
\leq 2\big(\|\Delta_{q}f\|_{L^{p}}^{2}+\|\Delta_{q}\rho\|_{L^{p}}^{2}+
\|R_{q}(u,\rho)\|_{L^p}\|\Delta_{q}\rho\|_{L^p}\big).
\end{equation*}
By using the commutator estimate \eqref{c-1}, for any $2\leq q<\infty$, we have
\begin{equation}\label{eq.CE}
 \norm{R_q(u,\rho)}_{L^p}
\leq C \big(\|S_{q+5}\nabla u\|_{L^\infty}\sum_{|q'-q|\leq5}\norm{\Delta_{q'}\rho}_{L^p}+2^{-q\sigma_{\tau}}\sqrt{q+2}
\norm{\omega}_{\sqrt{\mathbb{L}}}\|\rho\|_{B_{p,\infty}^{\sigma_{\tau}}}\big)
\end{equation}
Thanks to the divergence-free condition, we see that
\begin{equation}\label{eq.DE}
\begin{split}
\|S_{q+5}\nabla u\|_{L^\infty}\leq& C\|S_{q+5}\partial_{z}u\|_{L^\infty}+C\|u_r/r\|_{L^\infty}+C\|S_{q+5}\omega\|_{L^\infty}\\
\leq& C\|S_{q+5}\partial_{z}u\|_{L^\infty}+C\|u_r/r\|_{L^\infty}+C\sqrt{q}\|\omega\|_{\sqrt{\mathbb{L}}},
\end{split}
\end{equation}
where we have used the fact that
\begin{equation*}
\|S_{q}f\|_{L^{\infty}}\leq C2^{\frac{2q}{q}}\|f\|_{L^q}\leq C\sqrt{q}\norm{f}_{\sqrt{\mathbb{L}}}.
\end{equation*}
Plugging \eqref{eq.DE} in \eqref{eq.CE} gives
\begin{equation}\label{eq.CEE}
\begin{split}
&\norm{R_q(u,\rho)}_{L^p}\\
 \leq&C\big(\|S_{q+5}\partial_{z}u\|_{L^\infty}+\|u_r/r\|_{L^\infty}+\sqrt{q}\|\omega\|_{\sqrt{\mathbb{L}}}\big)\sum_{|q'-q|\leq5}\norm{\Delta_{q'}\rho}_{L^p} +C2^{-q\sigma_{\tau}}\sqrt{q+2}
\norm{\omega}_{\sqrt{\mathbb{L}}}\|\rho\|_{B^{\sigma_{\tau}}_{p,\infty}} \\
 \leq &C\big(\|S_{q+5}\partial_{z}u\|_{L^\infty}+\|u_r/r\|_{L^\infty}\big)\sum_{|q'-q|\leq5}\norm{\Delta_{q'}\rho}_{L^p} +C2^{-q\sigma_{\tau}}\sqrt{q+2}
\norm{\omega}_{\sqrt{\mathbb{L}}}\|\rho\|_{B^{\sigma_{\tau}}_{p,\infty}}.
\end{split}
\end{equation}
Therefore, we  get
\begin{equation}\label{l-1}
\begin{split}
\frac{\mathrm{d}}{\mathrm{d}t}\|\Delta_{q}\rho(t)\|_{L^p}^{2}\leq & C\Big(\|\Delta_{q}f\|_{L^{p}}^{2}+\|\Delta_{q}\rho\|_{L^p}^{2}+
2^{-q\sigma_{\tau}}\sqrt{q+2} \norm{\omega}_{\sqrt{\mathbb{L}}}\|\rho\|_{B_{p,\infty}^{\sigma_{\tau}}}
\|\Delta_{q}\rho\|_{L^p}\\&+\big(\|S_{q+5}\partial_{z}u\|_{L^\infty}+\|u_r/r\|_{L^\infty}\big)\sum_{|q'-q|\leq5}\norm{\Delta_{q'}\rho}_{L^p}
\|\Delta_{q}\rho\|_{L^p}\Big).
\end{split}
\end{equation}
By using  Young's inequality,  we easily find that
\begin{equation}\label{l-6}
\begin{split}
&\big(\|S_{q+5}\partial_{z}u\|_{L^\infty}+\|u_r/r\|_{L^\infty}\big)\sum_{|q'-q|\leq5}\norm{\Delta_{q'}\rho}_{L^p}
\|\Delta_{q}\rho\|_{L^p}\\
\leq &C\Big(\frac{\norm{\partial_{z}u}_{L^q}^{2}}{q^{3/4}}+\|u_r/r\|_{L^\infty}^2\Big)\|\Delta_{q}\rho\|_{L^p}^{2}+
Cq^{3/4}\sum_{|q'-q|\leq5}\|\Delta_{q'}\rho\|_{L^p}^{2}.
\end{split}
\end{equation}
Recall that, for any $t\in[0,T]$, $$U(t)=\exp{\Big(\sup_{2\leq q<\infty}\int_{0}^{t}
C\Big(1+\frac{\norm{\partial_{z}u(\tau)}_{L^q}^{2}}{q^{3/4}}\Big)\,\mathrm{d}\tau+C\int_0^t\Big\|\frac{u_r}{r}(\tau)\Big\|^2_{L^\infty}\,\mathrm{d}\tau\Big)}.$$
Plugging  bound \eqref{l-6} into \eqref{l-1} and then performing the Gronwall inequality to the resulting inequality, we readily have
\begin{equation}\label{l-2}
\begin{split}
\|\Delta_{q}\rho(t)\|_{L^p}^{2}\leq &U(t)
\Big(\|\Delta_{q}\rho_{0}\|_{L^p}^{2}+
C\int_{0}^{t}\|\Delta_{q}f(\tau)\|_{L^{p}}^{2}\,\mathrm{d}\tau
\\&+C\int_{0}^{t}2^{-q\sigma_{\tau}}\sqrt{q+2} \norm{\omega(\tau)}_{\sqrt{\mathbb{L}}}\|\rho(\tau)\|_{B_{p,\infty}^{\sigma_{\tau}}}
\|\Delta_{q}\rho(\tau)\|_{L^p}\,\mathrm{d}\tau\\
&+C\int_{0}^{t}(q+2)^{\frac34}\sum_{|q'-q|\leq5}\|\Delta_{q'}\rho(\tau)\|_{L^p}^{2}\,\mathrm{d}\tau\Big) \\:=&J_{0}+J_{1}+J_{2}+J_{3}.
\end{split}
\end{equation}
Multiplying \eqref{l-2} by $2^{2(q+2)\sigma_{t}}$, we have
\begin{equation}\label{l-7}
 2^{2(q+2)\sigma_{t}}\|\Delta_{q}\rho(t)\|_{L^p}^{2}\leq 2^{2(q+2)\sigma_{t}}(J_{0}+J_{1}+J_{2}+J_{3}).
 \end{equation}
For the first term on the left hand side of \eqref{l-7}, we have
\begin{equation*}
\begin{split}
2^{2(q+2)\sigma_{t}}J_{0}\leq &CU(t)2^{2(q+2)\sigma}
\|\Delta_{q}\rho_{0}\|_{L^p}^{2}2^{-2\eta(2+q)
\int_{0}^{t}V(\tau)\,\mathrm{d}\tau}\\ \leq &CU(t)\norm{\rho_{0}}_{B_{p,\infty}^{\sigma}}^{2}.
\end{split}
\end{equation*}
For the second term on the left hand side of \eqref{l-7}, we have
\begin{equation*}
\begin{split}
2^{2(q+2)\sigma_{t}}J_{1}\leq &C U(t)\int_{0}^{t}2^{2(q+2)\sigma_{\tau}}\|\Delta_{q}f(\tau)\|_{L^p}^{2}2^{-2\eta(2+q)
\int_{\tau}^{t}V(\tau')\,\mathrm{d}\tau'}\,\mathrm{d}\tau\\ \leq &CU(t)\int_{0}^{t}\norm{f(\tau)}_{B_{p,\infty}^{\sigma_{\tau}}}^{2}\,\mathrm{d}\tau.
\end{split}
\end{equation*}
For $2^{2(2+q)\sigma_{t}}J_{2}$, we get that
if  $q+2\geq \Big(\frac{2CU(T)}{\eta\log2}\Big)^{2}$,
\begin{equation*}
\begin{split}
2^{2(2+q)\sigma_{t}}J_{2}\leq&CU(t)\int_{0}^{t}\sqrt{q+2}
\norm{\omega(\tau)}_{\sqrt{\mathbb{L}}} 2^{-2\eta(2+q)
\int_{\tau}^{t}V(\tau')\,\mathrm{d}\tau'}
\|\rho(\tau)\|_{B_{p,\infty}^{\sigma_{\tau}}}
2^{q\sigma_{\tau}}\|\Delta_{q}\rho(\tau)\|_{L^p}\,\mathrm{d}\tau\\ \leq & \frac14\sup_{0\leq\tau\leq t}\|\rho(\tau)\|_{B_{p,\infty}^{\sigma_{\tau}}}^{2}.
\end{split}
\end{equation*}
Similarly, when $q+2\geq \Big(\frac{2CU(T)}{\eta\log2}\Big)^{4}$ we get that
\begin{equation*}
\begin{split}
2^{2(2+q)\sigma_{t}}J_{3}\leq&C\int_{0}^{t}(q+2)^{\frac34} 2^{-2\eta(2+q)
\int_{\tau}^{t}V(\tau')\,\mathrm{d}\tau'}
2^{2q\sigma_{\tau}}\|\Delta_{q}\rho(\tau)\|_{L^p}^{2}\,\mathrm{d}\tau
\\ \leq & \frac14\sup_{0\leq\tau\leq t}\|\rho(\tau)\|_{B_{p,\infty}^{\sigma_{\tau}}}^{2}.
\end{split}
\end{equation*}
On the other hand, it is noted that if $$q+2<\Big(\frac{2CU(T)}{\eta\log2}\Big)^{4},$$ both terms $2^{2(2+q)s_{t}}J_{2}$, $2^{2(2+q)s_{t}}J_{3}$ can be bounded by $$C\Big(\frac{U(T)}{\eta}\Big)^{3}\int_{0}^{t}V(\tau)
\|\rho(\tau)\|_{B_{p,\infty}^{\sigma_{\tau}}}^{2}\,\mathrm{d}\tau.$$
So finally, taking the supremum over $q\geq-1$ in \eqref{l-7} and using
 these above estimates, we have
\begin{equation*}
\begin{split}
&\|\rho(t)\|_{B_{p,\infty}^{\sigma_{t}}}^{2}\\ \leq &CU(t)\|\rho_{0}\|_{B_{p,\infty}^{\sigma}}^{2}+CU(t)\int_{0}^{t}
\|f(\tau)\|_{B_{p,\infty}^{\sigma_{\tau}}}^{2}\,\mathrm{d}\tau
+C\Big(\frac{U(T)}{\eta}\Big)^{3}\int_{0}^{t}V(\tau)
\|\rho(\tau)\|_{B_{p,\infty}^{\sigma_{\tau}}}^{2}\,\mathrm{d}\tau.
\end{split}
\end{equation*}
Performing the Gronwall inequality and using the definition of $\eta$, we eventually get  the desired losing estimates.
\end{proof}
Based on this proposition, we can get a similar result for the anisotropy transport-diffusion system~\eqref{losingestimate}.  More precisely, we have
\begin{prop}\label{losingestimate--1}
Let $\sigma\in (-1,1)$ and  $p\in [2,\infty)$. Assume that  $\rho$ is a smooth solution to system~\eqref{losingestimate} with $\rho_{0}\in B_{p,\infty}^{\sigma}$, $f\in L^{2}_{T}(B_{p,\infty}^{\sigma})$ and $g\in L^{2}_{T}(B_{p,\infty}^{\sigma})$ and  $u$ satisfies the same conditions as in Proposition \ref{losingestimate-1}.
Then, the following estimates hold for all small enough $\epsilon$:
\begin{equation*}
\sup_{0\leq t\leq T}\|\rho(t)\|^{2}_{B_{p,\infty}^{\sigma_{t}}}\leq CU(T)e^{\frac{CU^3(T)}{\epsilon^{3}}\big(\int_{0}^{T}V(t)\,\mathrm{d}t\big)^{4}}\left(\|\rho_{0}\|_{B_{p,\infty}^{\sigma}}^{2}+
\int_{0}^{T}\|f(\tau)\|_{B_{p,\infty}^{\sigma_{\tau}}}^{2}+
\|g(\tau)\|_{B_{p,\infty}^{\sigma_{\tau}}}^{2}\,\mathrm{d}\tau\right).
\end{equation*}
In particular, we have for all small enough $\epsilon>0$:
\begin{equation*}
\|\rho\|_{L^{\infty}_{T}(B_{p,\infty}^{\sigma-\epsilon})}\leq CU^{\frac12}(T)e^{\frac{CU^3(T)}{\epsilon^{3}}\big(\int_{0}^{T}V(t)
\,\mathrm{d}t\big)^{4}}\big(\|\rho_{0}\|_{B_{p,\infty}^{\sigma}}
+\|f\|_{L^{2}_{T}(B_{p,\infty}^{\sigma})}+\|g\|_{L^{2}_{T}(B_{p,\infty}^{\sigma})}\big).
\end{equation*}
Here the constant $C>0$ depends  on $p$, $\sigma$, $T$, and  $V(T)$, $U(T)$ are defined as in Proposition \ref{losingestimate-1}.
\end{prop}

\begin{proof}
Applying  $\Delta_{q}$ to the first equation of system \eqref{losingestimate}. With the notation introduced in Lemma~\ref{commutator-est}, we have  $$\partial_{t}\Delta_{q}\rho+(S_{q-1}u\cdot\nabla)\Delta_{q}\rho-\nu\partial_{zz}^{2}\Delta_{q}\rho=
\Delta_{q}f+\partial_{z}\Delta_{q}g+R_{q}(u,\rho).$$
Multiplying this inequality by  $|\Delta_{q}\rho|^{p-2}\Delta_{q}\rho$, $2\leq p<\infty$ and integrating the resulting equation, we get
\begin{equation}\label{losingestimate-7}
\begin{split}
&\frac{1}{p}\frac{\mathrm{d}}{\mathrm{d}t}\|\Delta_{q}\rho(t)\|_{L^{p}}^{p}
+\frac{4(p-1)}{p^{2}}\big\|\partial_{z}|\Delta_{q}\rho|^{\frac{p}{2}}(t)
\big\|_{L^{2}}^{2}\\=&
\int_{\mathbb{R}^3}\Delta_{q}f |\Delta_{q}\rho|^{p-2}\Delta_{q}\rho \,\mathrm{d}\mathrm{\mathbf{x}}+\int_{\mathbb{R}^3} R_{q}(u,\rho) |\Delta_{q}\rho|^{p-2}\Delta_{q}\rho \,\mathrm{d}\mathrm{\mathbf{x}}-(p-1)\int_{\mathbb{R}^3}\Delta_{q}g|\Delta_{q}\rho|^{p-2}\partial_{z}\Delta_{q}\rho\, \mathrm{d}\mathrm{\mathbf{x}}\\ :=&I_{1}+I_{2}+I_{3}.
\end{split}
\end{equation}
By the H\"{o}lder inequality, we have
$$I_{1}\leq \|\Delta_{q}f\|_{L^{p}}\|\Delta_{q}\rho\|_{L^{p}}^{p-1}.$$
In a similar fashion as \eqref{eq.CEE}, one can conclude
\begin{equation*}
\begin{split}
I_{2}\leq &\|R_{q}(u,\rho)\|_{L^{p}}\|\Delta_{q}\rho\|_{L^{p}}^{p-1}\\ \leq&
C\big(\|S_{q+5}\partial_{z}u\|_{L^\infty}+\|u_r/r\|_{L^\infty}\big)\sum_{|q'-q|\leq5}\norm{\Delta_{q'}\rho}_{L^p} \|\Delta_{q}\rho\|_{L^p}^{p-1}\\
&+C2^{-q\sigma_{\tau}}\sqrt{q+2}
\norm{\omega}_{\sqrt{\mathbb{L}}}\|\rho\|_{B^{\sigma_{\tau}}_{p,\infty}}\|\Delta_{q}\rho\|_{L^p}^{p-1}.
\end{split}
\end{equation*}
We use the H\"{o}lder inequality and the Young inequality to estimate $I_{3}$ as follows
\begin{equation*}
\begin{split}
I_{3}\leq &\frac{2(p-1)}{p}\|\Delta_{q}g\|_{L^{p}}\|\Delta_{q}\rho\|_{L^{p}}^{\frac{p-2}{2}}
\big\|\partial_{z}|\Delta_{q}\rho|^{\frac{p}{2}}\big\|_{L^{2}}
\\ \leq& \frac{2(p-1)}{p^{2}}\big\|\partial_{z}|\Delta_{q}\rho|^{\frac{p}{2}}\big\|_{L^{2}}^{2}+
Cp\|\Delta_{q}g\|_{L^{p}}^{2}\|\Delta_{q}\rho\|_{L^{p}}^{p-2}.
\end{split}
\end{equation*}
Plugging the last three estimates into \eqref {losingestimate-7}, we immediately get
\begin{equation*}
\begin{split}
\frac{\mathrm{d}}{\mathrm{d}t}\|\Delta_{q}\rho(t)\|_{L^{p}}^{2}\leq &
 C\big(\|\Delta_{q}f\|_{L^{p}}^{2}+\|\Delta_{q}\rho\|_{L^p}^{2}+
2^{-q\sigma_{\tau}}\sqrt{q+2} \norm{\omega}_{\sqrt{\mathbb{L}}}\|\rho\|_{B_{p,\infty}^{\sigma_{\tau}}}
\|\Delta_{q}\rho\|_{L^p}\\&+C\Big(\frac{\norm{\partial_{z}u}_{L^q}^{2}}{q^{3/4}}+\|u_r/r\|_{L^\infty}^2\Big)\|\Delta_{q}\rho\|_{L^p}^{2}+
Cq^{3/4}\sum_{|q'-q|\leq5}\|\Delta_{q'}\rho\|_{L^p}^{2}+p\|\Delta_{q}g\|_{L^{p}}^{2}\big).
\end{split}
\end{equation*}
It is now easy to conclude the desired result of this proposition. In fact, it is just a matter of arguing exactly as in Proposition \ref{losingestimate-1}.
\end{proof}

\section{A priori estimates}\label{Sec-3}
In this section, we aim at establishing the global a priori estimates needed for the proof of Theorem~\ref{the}. We  first prove the natural  energy estimates associated to system \eqref{bs}. In the second step, we  present the control of some stronger norms such as  $\|\omega\|_{L^{\infty}_{t}(\sqrt{\mathbb{L}})}$ and $\displaystyle\sup_{2\leq p<\infty}\int_0^t\frac{\norm{\partial_{z} u(\tau)}^{2}_{L^p}}{p^{3/4}}\,\mathrm{d}\tau$. In the third step, we  prove the global Lipschitz estimates of the vector field $u$ by making good use of losing estimates. Finally, with the  help of the special structure of  system \eqref{bs} and commutator estimates in Lemma \ref{commutator-est}, we  show  the $H^{s}\times H^{s}$, $s>\frac52$ a priori estimates of $(u,b)$.
\subsection{The natural energy estimates}
Now, let us begin with the natural energy estimates of $(u,b)$.
\begin{prop}\label{energy}
Assume that $(u_{0}, b_{0})\in L^{2}\times L^{2}$. Let $(u, b)$  be the smooth solution of system~\eqref{bs}. Then, for any $t\geq0$, there holds
\begin{equation}\label{energy}
\|u(t)\|_{L^2}^{2}+\|b(t)\|_{L^2}^{2}+2\int_{0}^{t}\|\partial_{_{z}}u(\tau)\|_{L^2}^{2}\,\mathrm{d}\tau\leq
\|u_{0}\|_{L^2}^{2}+\|b_{0}\|_{L^2}^{2}.
\end{equation}
\end{prop}
\begin{proof}
Although the proof of this proposition is standard, we  give the proof for reader's convenience. Taking the $L^2$ inner product of the velocity equation with $u$ and using the divergence-free condition of $u$,  we find that
\begin{equation}\label{energy-1}
\frac12\frac{\mathrm{d}}{\mathrm{d}t}\|u(t)\|_{L^2}^2+
\|\partial_{z}u(t)\|_{L^2}^2
=\int_{\mathbb{R}^3}\big((b\cdot\nabla) b\big)\cdot u \,\mathrm{d}\mathrm{\mathbf{x}}.
\end{equation}
In a similar way, we can get $L^2$-estimate of $b$:
\begin{equation}\label{energy-2}
\frac12\frac{\mathrm{d}}{\mathrm{d}t}\|b(t)\|_{L^2}^2
=\int_{\mathbb{R}^3} \big((b\cdot\nabla) u\big)\cdot b\,\mathrm{d}\mathrm{\mathbf{x}}.
\end{equation}
Note that
\[\int_{\mathbb{R}^3} \big((b\cdot\nabla) b\big)\cdot u \,\mathrm{d}\mathrm{\mathbf{x}}+\int_{\mathbb{R}^3} \big((b\cdot\nabla)u\big)\cdot b\,\mathrm{d}\mathrm{\mathbf{x}}=0.\]
This together with  \eqref{energy-1} and  \eqref{energy-2}  yields
\[\frac12\frac{\mathrm{d}}{\mathrm{d}t}\left(\|u(t)\|_{L^2}^2+\|b(t)\|^2_{L^2}\right)+
\|\partial_{z}u(t)\|_{L^2}^2
=0.\]
Integrating the  equality with respect to time $t$ leads to the desired  estimate.
\end{proof}
Let us point out  that the axisymmetric  assumption is not needed in the proposition above. However, we need to use the structural assumptions to show the strong estimates of $(u,b)$. We always assume that
\begin{equation}\label{eq.axi-ass}
u=u_{r}e_{r}+u_{z}e_{z} \quad \text{and} \quad b=b_{\theta}e_{\theta}
\end{equation} in the remainder parts of this section.
\subsection{Strong a priori estimates}
This subsection is devoted to obtaining some strong a priori estimates of $(u,b)$. Let us start with the maximum principle of  quantity  $\frac{b_\theta}{r}$ which solves the homogenous transport equation.
\begin{prop}\label{b}
Let $\frac{b_{0}}{r}\in L^{2}\cap L^{\infty}$. Assume that $(u, b)$ is   the smooth solution of  system~\eqref{bs} satisfying \eqref{eq.axi-ass}. Then,  for all $p\in [2,\infty]$, there holds
$$\Big\|\frac{b_{\theta}(t)}{r}\Big\|_{L^p}\leq \Big\|\frac{b_{0}}{r}\Big\|_{L^{2}\cap L^{\infty}}\quad\text{for any}\quad t\geq0.
$$
\end{prop}

\begin{proof}
Recall that
 $$ \partial_{t}b_{\theta}+(u\cdot\nabla)b_{\theta}=\frac{u_{r}}{r}b_{\theta},$$
 we easily find that $\frac{b_{\theta}}{r}$ satisfies the following homogeneous transport equation
$$ \partial_{t}\Big(\frac{b_{\theta}}{r}\Big)+(u\cdot\nabla)\Big( \frac{b_{\theta}}{r}\Big)=0.$$
Moreover, by the divergence-free condition of $u$, we have
$$\Big\|\frac{b_{\theta}}{r}(t)\Big\|_{L^p}\leq \Big\|\frac{b_{0}}{r}\Big\|_{L^p},  \quad\text{for all}\quad p\in [2,\infty].
$$
Therefore, we finish the proof of this proposition by using the interpolation inequality.
\end{proof}
Next, we show the $L^{3,1}$-estimate of the physical quantity  $\frac{\omega_{\theta}}{r}$, which plays an important role in the study of axisymmetric flow without swirl.
\begin{prop}\label{basic}
Assume that $\frac{\omega_{0}}{r}\in L^{p,q}$ with $1<p<\infty,\,1\leq q\leq\infty$ and $\frac{b_{0}}{r}\in L^{2}\cap L^{\infty}$.  Let $(u, b)$ be  the smooth solution of  system \eqref{bs} satisfying \eqref{eq.axi-ass}. Then, for any $t\geq0$,  the following estimate holds
\begin{equation*}
\Big\|\frac{\omega_{\theta}}{r}(t)\Big\|_{L^{p,q}}\leq
C\Big(\Big\|\frac{\omega_{0}}{r}\Big\|_{L^{p,q}}+\sqrt{t}
\Big\|\frac{b_{0}}{r}\Big\|^{2}_{L^{2}\cap L^{\infty}}\Big).
\end{equation*}
In particular, we have
\begin{equation}\label{prior-w-1}
\Big\|\frac{\omega_{\theta}}{r}(t)\Big\|^2_{L^2}+\int_0^t\Big\|\partial_z\Big(\frac{\omega_{\theta}}{r}\Big)(s)\Big\|^2_{L^2}\,\mathrm{d}s\leq C\Big(\Big\|\frac{\omega_{0}}{r}\Big\|^2_{L^2}+t\Big\|\frac{b_{0}}{r}\Big\|^{4}_{L^{2}\cap L^{\infty}}\Big).
\end{equation}
\end{prop}
\begin{proof}
We observe  that the quantity $\Gamma:=\frac{\omega_{\theta}}{r}$ solves the following equation
\begin{equation}\label{r}
\partial_{t}\Gamma+(u\cdot\nabla) \Gamma-\partial_{zz}^{2}\Gamma=-\frac{\partial_{z}(b_{\theta}^{2})}{r^{2}}.
\end{equation}
Multiplying \eqref{r} by $|\Gamma|^{p-2}\Gamma$ with $2\leq p<\infty$ and integrating the resulting equation over $\mathbb{R}^3,$ we readily have
\begin{equation}\label{r0}
\begin{split}
\frac1p\frac{\mathrm{d}}{\mathrm{d}t}\|\Gamma(t)\|_{L^p}^p+\frac{4(p-1)}{p^{2}}
\big\|\partial_{z}|\Gamma(t)|^{\frac{p}{2}}\big\|_{L^2}^2
&=-\int_{\mathbb{R}^3} \partial_{z}\Big(\frac{b_{\theta}^{2}}{r^{2}}\Big)|\Gamma|^{p-2}\Gamma \,\mathrm{d}\mathrm{\mathbf{x}}
\\&=(p-1)\int_{\mathbb{R}^3} \frac{b_{\theta}^{2}}{r^{2}}|\Gamma|^{p-2}\partial_{z}\Gamma \,\mathrm{d}\mathrm{\mathbf{x}}.
\end{split}
\end{equation}
For the integral term in the last line of \eqref{r0}, by the H\"{o}lder inequality and the Young inequality, we get
\begin{equation*}
\begin{split}
(p-1)\int_{\mathbb{R}^3}\frac{b_{\theta}^{2}}{r^{2}}|\Gamma|^{p-2}\partial_{z}\Gamma\,\mathrm{d}\mathrm{\mathbf{x}}&\leq (p-1)\Big\|\Big(\frac{b_{\theta}}{r}\Big)^{2}\Big\|_{L^p}\big\||\Gamma|^{\frac{p-2}{2}}\big\|_{L^{\frac{2p}{p-2}}}
\big\||\Gamma|^{\frac{p-2}{2}}\partial_{z}\Gamma \big\|_{L^2}\\&
\leq\frac{2(p-1)}{p}\Big\|\frac{b_{\theta}}{r}\Big\|_{L^{2p}}^{2}\|\Gamma\|_{L^{p}}^{\frac{p-2}{2}}
\big\|\partial_{z}|\Gamma|^{\frac{p}{2}}\big\|_{L^2}\\
&\leq \frac{2(p-1)}{p^{2}}\big\|\partial_{z}|\Gamma|^{\frac{p}{2}}\big\|_{L^2}^{2}+Cp\Big\|\frac{b_{\theta}}{r}\Big\|_{L^{2p}}^{4}\|\Gamma\|_{L^{p}}^{p-2}.
\end{split}
\end{equation*}
Combining this with \eqref{r0} and Proposition \ref{b} gives that for all $p\in[2,\infty)$,
\begin{equation}\label{prior-w}
\begin{split}
\frac1p\frac{\mathrm{d}}{\mathrm{d}t}\|\Gamma(t)\|_{L^p}^p+\frac{2(p-1)}{p^{2}}
\big\|\partial_{z}|\Gamma(t)|^{\frac{p}{2}}\big\|_{L^2}^2
\leq&Cp\Big\|\frac{b_{\theta}}{r}\Big\|_{L^{2p}}^{4}\|\Gamma\|_{L^{p}}^{p-2}\\
\leq&Cp\Big\|\frac{b_{0}}{r}\Big\|_{L^{2p}}^{4}\|\Gamma\|_{L^{p}}^{p-2}\\
\leq&Cp\Big\|\frac{b_{0}}{r}\Big\|_{L^{2}\cap L^\infty}^{4}\|\Gamma\|_{L^{p}}^{p-2}.
\end{split}
\end{equation}
It follows that for all $p\in[2,\infty)$,
\begin{equation*}
\frac{\mathrm{d}}{\mathrm{d}t}\|\Gamma(t)\|_{L^p}^2\leq Cp\Big\|\frac{b_{0}}{r}\Big\|_{L^{2}\cap L^{\infty}}^{4}.
\end{equation*}
After integrating  this inequality with respect to  time $t$, we get
\begin{equation}\label{add-w}
\|\Gamma(t)\|_{L^p}^2\leq \|\Gamma(0)\|_{L^{p}}^2+Cpt\Big\|\frac{b_{0}}{r}\Big\|_{L^{2}\cap L^{\infty}}^{4}.
\end{equation}
Therefore, by the interpolation theorem, we finally have
\begin{equation*}
\|\Gamma(t)\|_{L^{p,q}}\leq
C\Big(\|\Gamma(0)\|_{L^{p,q}}+\sqrt{t}
\Big\|\frac{b_{0}}{r}\Big\|^{2}_{L^{2}\cap L^{\infty}}\Big).
\end{equation*}
Choosing $p=2$ in \eqref{prior-w} and integrating with respect to time $t$ yield estimate \eqref{prior-w-1}, and thus we completes the proof of Proposition \ref{basic}.
\end{proof}

Based on the estimate of $\frac{\omega_\theta}{r}$, we further establish the estimate of $b$ and vorticity $\omega$. Firstly, we focus on the maximum principle of $b_\theta$.

\begin{prop}\label{basic2}
Assume that $\frac{\omega_{0}}{r}\in L^{3,1}$, $b_{0}\in L^{2}\cap L^{\infty}$ and  $\frac{b_{0}}{r}\in L^{2}\cap L^{\infty}$. Let $(u, b)$ be  the smooth solution of  system \eqref{bs} satisfying \eqref{eq.axi-ass}.  Then, for any $t\geq0$, there holds
\begin{equation*}
\norm{b_\theta(t)}_{L^{p}}\leq e^{C\big(t\big\|\frac{\omega_{0}}{r}\big\|_{L^{3,1}}+t^{\frac{3}{2}}
\big\|\frac{b_{0}}{r}\big\|^{2}_{L^{2}\cap L^{\infty}}\big)}
\norm{b_{0}}_{L^{2}\cap L^{\infty}}, \quad\text{for each}\quad p\in [2,\infty].
\end{equation*}
\end{prop}

\begin{proof}
Multiplying  the third equation of system \eqref{MHDaxi} by $|b_\theta|^{p-2}b_\theta$, $2\leq p<\infty$ and performing integration in space, we get
\begin{equation*}
\begin{split}
\frac1p\frac{\mathrm{d}}{\mathrm{d}t} \|b_\theta(t)\|_{L^{p}}^{p}
=&\int_{\mathbb{R}^3}\frac{u_r}{r} |b_\theta|^p\,\mathrm{d}\mathrm{\mathbf{x}}.
\end{split}
\end{equation*}
For  the right hand side term, we deduce by the H\"older inequality that
\begin{equation*}
\int_{\mathbb{R}^3}\frac{u_r}{r} |b_\theta|^p\mathrm{d}\mathrm{\mathbf{x}}\leq \Big\|\frac{u_{r}}{r}\Big\|_{L^{\infty}}\norm{b_\theta}^{p}_{L^p}.
\end{equation*}
Therefore,
\begin{equation*}
\frac{\mathrm{d}}{\mathrm{d}t}\norm{b_\theta(t)}_{L^{p}}\leq \Big\|\frac{u_{r}}{r}\Big\|_{L^{\infty}}\norm{b_\theta}_{L^p}.
\end{equation*}
The Gronwall lemma yields that
\begin{equation}\label{r6}
\norm{b_\theta(t)}_{L^{p}}\leq e^{\int_0^t \|\frac{u_{r}}{r}(\tau)\|_{L^{\infty}}\,\mathrm{d}\tau}
\norm{b_{\theta}(0)}_{L^{p}}.
\end{equation}
To estimate $\big\|\frac{u_{r}}{r}\big\|_{L^{1}_{t}(L^{\infty})}$, we  use the the pointwise estimate
$\big|\frac{u_{r}}{r}\big|\leq\frac{1}{|\cdot|^{2}}\ast\big|\frac{\omega_{\theta}}{r}\big|$
and Lemma \ref{con} to obtain
\begin{equation}\label{basic0}
\Big\|\frac{u_{r}}{r}\Big\|_{L^{\infty}}\leq \Big\|\frac{\omega_{\theta}}{r}\Big\|_{L^{3,1}}.
\end{equation}
By inserting estimates \eqref{basic0} into \eqref{r6} and using Proposition \ref{basic}, we get
\begin{equation}\label{prior-w-2}
\norm{b_\theta(t)}_{L^{p}}\leq e^{C\big(t\big\|\frac{\omega_{0}}{r}\big\|_{L^{3,1}}+t^{\frac{3}{2}}
\big\|\frac{b_{0}}{r}\big\|^{2}_{L^{2}\cap L^{\infty}}\big)}
\norm{b_{0}}_{L^{p}}.
\end{equation}
For  $p=\infty$, we see that
\begin{equation*}
\begin{split}
\norm{b_\theta(t)}_{L^{\infty}}\leq &\norm{b_\theta(0)}_{L^{\infty}}+\int_0^t\Big\|\Big(\frac{u_r}{r} b_\theta\Big)(s)\Big\|_{L^\infty}\,\mathrm{d}s\\
\leq &\norm{b_0}_{L^{\infty}}+
\int_0^t\Big\|\frac{u_r(s)}{r}\Big\|_{L^\infty}\|b_\theta(s)\|_{L^\infty}\,\mathrm{d}s.
\end{split}
\end{equation*}
By the similar argument as above, we deduce that
\begin{equation*}
\norm{b_\theta(t)}_{L^{\infty}}\leq e^{C\big(t\big\|\frac{\omega_{0}}{r}\big\|_{L^{3,1}}+t^{\frac{3}{2}}
\big\|\frac{b_{0}}{r}\big\|^{2}_{L^{2}\cap L^{\infty}}\big)}
\norm{b_{0}}_{L^{\infty}}.
\end{equation*}
This combined with \eqref{prior-w-2} yields the desired result.
\end{proof}

With the estimates established above in hand, we can bound $\norm{\omega_{\theta}}_{L^{\infty}_{t}(\sqrt{{\mathbb{L}}})}$ by using the smoothing effect of the vertical diffusion.

\begin{prop}\label{pro1}
Assume  that $b_{0}\in {L^{2}\cap L^{\infty}}$, $\frac{b_{0}}{r}\in {L^{2}\cap L^{\infty}}$, $\omega_{0}\in \sqrt{{\mathbb{L}}}$
and $\frac{\omega_{0}}{r}\in L^{3,1}$.
If $(u, b)$ is  the smooth solution of system \eqref{bs} satisfying \eqref{eq.axi-ass},  then for any $t\geq0$, there holds
\begin{equation}\label{eq.voriticity-L}
\norm{\omega_{\theta}(t)}_{\sqrt{{\mathbb{L}}}}\leq C.
\end{equation}
Here the constant C depends only on $t$, $\norm{b_{0}}_{L^{2}\cap L^{\infty}}$, $\big\|\frac{b_{0}}{r}\big\|_{L^{2}\cap L^{\infty}}$, $\norm{\omega_{0}}_{\sqrt{{\mathbb{L}}}}$
and $\big\|\frac{\omega_{0}}{r}\big\|_{L^{3,1}}$.
\end{prop}

\begin{proof}
Multiplying the vorticity equation \eqref{w} by $|\omega_\theta|^{p-2}\omega_\theta$, $2\leq p<\infty$ and integrating the resulting equation over $\mathbb{R}^3$, we obtain
\begin{equation}\label{r2}
\frac1p\frac{\mathrm{d}}{\mathrm{d}t}\|\omega_\theta(t)\|_{L^p}^p+\frac{4(p-1)}{p^{2}}
\big\|\partial_{z}|\omega_\theta(t)|^{\frac{p}{2}}\big\|_{L^2}^2
=\int_{\mathbb{R}^3}\frac{u_{r}}{r}|\omega_\theta|^{p} \,\mathrm{d}\mathrm{\mathbf{x}}-\int_{\mathbb{R}^3}\frac{\partial_{z}(b_{\theta}^{2})}{r}|\omega_\theta|^{p-2} \omega_\theta\,\mathrm{d}\mathrm{\mathbf{x}}.
\end{equation}
For the first term on the right hand side of the above equality, we deduce by the H\"older inequality that
\begin{equation}\label{r3}
\int_{\mathbb{R}^3}\frac{u_r}{r} |\omega_\theta|^p\,\mathrm{d}\mathrm{\mathbf{x}}\leq \Big\|\frac{u_{r}}{r}\Big\|_{L^{\infty}}\norm{\omega_\theta}^{p}_{L^p}.
\end{equation}
As for the second term, integrating by parts leads to
\begin{align}\label{r7}
-\int_{\mathbb{R}^3}\frac{\partial_{z}(b_{\theta}^{2})}{r}|\omega_\theta|^{p-2} \omega_\theta\,\mathrm{d}\mathrm{\mathbf{x}}=&(p-1)\int_{\mathbb{R}^3}\frac{b_{\theta}^{2}}{r} |\omega_\theta|^{p-2}\partial_{z}\omega_\theta\,\mathrm{d}\mathrm{\mathbf{x}}.
\end{align}
Moreover, by the H\"older inequality and the Young inequality, we get
\begin{align*}
&(p-1)\int_{\mathbb{R}^3}\frac{b_{\theta}^{2}}{r} |\omega_\theta|^{p-2}\partial_{z}\omega_\theta\,\mathrm{d}\mathrm{\mathbf{x}}\\
\leq &(p-1)\Big\|\frac{b_{\theta}^{2}}{r} \Big\|_{L^p}\big\||\omega_\theta|^{\frac {p-2}{2}}\big\|_{L^\frac{2p}{p-2}}\big\||\omega_\theta|^{\frac {p-2}{2}}\partial_{z}\omega_\theta\big\|_{L^2}\\ \leq&\frac{2(p-1)}{p}\Big\|\frac{b_{\theta}^{2}}{r} \Big\|_{L^p}\|\omega_\theta\|_{L^p}^{\frac{p-2}{2}}
\big\|\partial_{z}|\omega_\theta\big|^{\frac {p}{2}}\|_{L^2}
\\ \leq&\frac{2(p-1)}{p^{2}}\big\|\partial_{z}|\omega_\theta|^{\frac {p}{2}}\big\|_{L^2}^{2}+Cp
\Big\|\frac{b_{\theta}^{2}}{r} \Big\|_{L^p}^{2}\|\omega_\theta\|_{L^p}^{p-2}.
\end{align*}
This together with \eqref{r2}, \eqref{r3} and \eqref{r7} yields
\begin{equation*}
\frac{\mathrm{d}}{\mathrm{d}t}\|\omega_\theta(t)\|_{L^p}^2\leq
C\Big\|\frac{u_{r}}{r}\Big\|_{L^{\infty}}\norm{\omega_\theta}^{2}_{L^p}+Cp\Big\|\frac{b_{\theta}^{2}}{r} \Big\|_{L^p}^{2}.
\end{equation*}
Hence, the Gronwall lemma ensures that
\begin{equation}\label{r4}
\|\omega_\theta(t)\|_{L^p}^2\leq
e^{C\int_0^t\|\frac{u_{r}}{r}(\tau)\|_{L^{\infty}}\,\mathrm{d}\tau}
\Big(\norm{\omega_{\theta}(0)}^{2}_{L^{p}}+
Cp\int_0^t\Big\|\frac{b_{\theta}^{2}}{r}(\tau)\Big\|^{2}_{L^{p}}\,\mathrm{d}\tau\Big).
\end{equation}
According to Proposition \ref{b} and Proposition \ref{basic2}, we obtain  that
\begin{equation*}
\begin{split}
\int_0^t\Big\|\frac{b_{\theta}^{2}(\tau)}{r}\Big\|^{2}_{L^{p}}\,\mathrm{d}\tau&\leq
\int_0^t\|b_{\theta}(\tau)\|^{4}_{L^{2p}}\,\mathrm{d}\tau+
\int_0^t\Big\|\frac{b_{\theta}(\tau)}{r}\Big\|^{4}_{L^{2p}}\,\mathrm{d}\tau\\&\leq
Cte^{C\Big(t\big\|\frac{\omega_{0}}{r}\big\|_{L^{3,1}}+t^{\frac{3}{2}}
\big\|\frac{b_{0}}{r}\big\|^{2}_{L^{2}\cap L^{\infty}}\Big)}
\norm{b_{0}}_{L^{2}\cap L^{\infty}}^{4}+Ct\Big\|\frac{b_{0}}{r}\Big\|^{4}_{L^2\cap L^\infty}\\&\leq C.
\end{split}
\end{equation*}
By virtue of \eqref{basic0} and Proposition \ref{basic}, we know that
$$\int_0^t\Big\|\frac{u_{r}}{r}(\tau)\Big\|_{L^{\infty}}\,\mathrm{d}\tau\leq C,$$ where the constant $C$ doesn't depend on $p$.
Inserting all these estimates into  \eqref{r4}, we get for each $2\leq p<\infty$,
$$\|\omega_\theta(t)\|_{L^p}\leq C(\|\omega_0\|_{L^p}+\sqrt{p}),$$ which implies the desired estimate \eqref{eq.voriticity-L}.
\end{proof}

Proposition \ref{pro1}  together with the well-known fact that $\|\nabla u\|_{L^p}\leq C\frac{p^2}{p-1}\|\omega\|_{L^p}$ for $p\in(1,\infty)$ yields that
$\displaystyle\sup_{2\leq p<\infty}\frac{\|\nabla u(t)\|_{L^p}}{p\sqrt{p}}$ is locally bounded in time. But, the growth rate $p\sqrt{p}$ goes far beyond  the Osgood type theorem. This induces us  to improve the regularity of $\nabla u$ by using the vertical smooth effect for the anisotropy system, which is the heart in our proof.

\begin{prop}\label{prop-fine}
Assume that $b_{0}\in {L^{2}\cap L^{\infty}}$, $\frac{b_{0}}{r}\in {L^{2}\cap L^{\infty}}$, $\omega_{0}\in \sqrt{{\mathbb{L}}}$
and $\frac{\omega_{0}}{r}\in L^{3,1}$.
Let $(u, b)$ be  the smooth  solution of  system \eqref{bs} satisfying \eqref{eq.axi-ass}. Then, for any $t\geq0$, the following  estimate holds
\begin{equation}\label{eq-h-e}
\sup_{2\leq p<\infty}\int_0^t\sum_{q\geq0}2^{qs}\frac{\|u_{q}(\tau)\|_{L^p}}{p^{\frac32}}\,\mathrm{d}\tau\leq C,\quad \text{for}\quad s\in(1,2).
\end{equation}
In particular, we have
\begin{equation}\label{Loglip}
\sup_{2\leq p<\infty}\int_0^t\frac{\norm{\partial_{z} u(\tau)}_{L^p}^{2}}{p^{\frac34}}\,\mathrm{d}\tau\leq C.
\end{equation}
Here  constants $C>0$ depend on  $t$, $\norm{b_{0}}_{L^{2}\cap L^{\infty}}$, $\big\|\frac{b_{0}}{r}\big\|_{L^{2}\cap L^{\infty}}$, $\norm{\omega_{0}}_{\sqrt{{\mathbb{L}}}}$
and $\big\|\frac{\omega_{0}}{r}\big\|_{L^{3,1}}$.
\end{prop}
\begin{proof}
Applying  operator $\Delta^{v}_{q}\mathcal{P}$ to the first equation of \eqref{bs} and using Duhamel formula,  we get
\begin{equation*}
\begin{split}
u_{q}(t,\mathrm{\mathbf{x}})=&e^{t\Delta_{v}}u_{q}(0)-\int_0^te^{(t-\tau)\Delta_{v}}\Delta^{v}_q\mathcal{P}\big((u\cdot\nabla) u\big)(\tau,\mathrm{\mathbf{x}})\,\mathrm{d}\tau+\int_0^te^{(t-\tau)\Delta_{v}}\Delta^{v}_q\mathcal{P}
\big((b\cdot\nabla) b\big)(\tau,\mathrm{\mathbf{x}})\,\mathrm{d}\tau,
\end{split}
\end{equation*}
where $u_{q}=\Delta^{v}_q u$ and $\mathcal{P}$ is the Leray projection on divergence free vector fields.

Notice that
\begin{equation*}
(u\cdot\nabla) u=\omega\times u+\frac12\nabla|u|^2
\end{equation*}
whence,
$$\mathcal{P}\big((u\cdot\nabla) u\big)=\mathcal{P}(\omega\times u).
$$
According to Lemma \ref{heat}, for $q\geq0$, we have
$$
\|e^{t\Delta_{v}}\Delta^{v}_q f\|_{L^p(\RR^3)}=\norm{\|e^{t\Delta_{v}}\Delta^{v}_q f\|_{L^p(\RR_{v})}}_{L^{p}(\RR_{h}^2)}\le Ce^{-ct2^{2q}}\|\Delta^{v}_q  f\|_{L^p(\RR^3)}.
$$
Therefore,  for $q\geq 0$,
\begin{equation*}
\begin{split}
\|u_{q}\|_{L^1_t(L^p)}\leq& C 2^{-2q}\|u_{q}(0)\|_{L^p}+Cp2^{-2q}\int_0^t\big\|\Delta^{v}_q (\omega\times u)(\tau)\big\|_{L^p}\,\mathrm{d}\tau\\
&+Cp2^{-2q}\int_0^t\big\|\Delta^{v}_{q}\big((b\cdot\nabla) b\big)(\tau)\big\|_{ L^p}\,\mathrm{d}\tau.
\end{split}
\end{equation*}
Multiplying the above inequality by $2^{qs}$ and summing over $q\geq0$, we readily obtain that
\begin{equation*}
\begin{split}
 \sum_{q\geq0}2^{qs}\|u_{q}\|_{L^1_t(L^p)}
\leq & C\sum_{q\geq0}2^{q(s-2)}\|u_{q}(0)\|_{L^p}+C
p\int_0^t\sum_{q\geq0}2^{q(s-2)}\big\|\Delta^{v}_q (\omega\times u)(\tau)\big\|_{L^p}\,\mathrm{d}\tau\\
&+C
p\int_0^t\sum_{q\geq0}2^{q(s-2)}\big\|\Delta^{v}_q \big((b\cdot\nabla) b\big)(\tau)\big\|_{L^p}\,\mathrm{d}\tau\\
:=&I_1+I_2+I_3.
\end{split}
\end{equation*}
First of all,  the H\"older inequality and the Sobelev inequality allow us to conclude that for $s<2$,
\begin{equation}\label{eq.I1}
\begin{split}
I_1\leq C\|u_0\|_{L^p}\leq& C(\|u_0\|_{L^2}+\|u_0\|_{L^\infty})\\
\leq& C(\|u_0\|_{L^2}+\|\omega_0\|_{\sqrt{\mathbb{L}}}).
\end{split}
\end{equation}
To deal with   $I_2$, arguing  as for proving  \eqref{eq.I1}, we get that  for $s<2$,
\begin{equation*}
\begin{split}
\sum_{q\geq0}2^{q(s-2)}\big\|\Delta^{v}_q (\omega\times u)\big\|_{L^p}\leq&C\|\omega\times u\|_{L^p}\\
\leq&C\|\omega\|_{L^p}\|u\|_{L^\infty}\\
\leq&C\|\omega\|_{L^p}\big(\|u\|_{L^2}+\|\omega\|_{L^4}\big).
\end{split}
\end{equation*}
In consequence,
\begin{equation}\label{eq.I2}
\begin{split}
I_{2}\leq &Cp\big(\|u_0\|_{L^2}+\|\omega\|_{L^\infty_t(\sqrt{\mathbb{L}})}\big)
\int_0^t\norm{\omega(\tau)}_{L^{p}}\,\mathrm{d}\tau\\ \leq&Ctp^{\frac{3}{2}}\big(\|u_0\|_{L^2}+\|\omega\|_{L^\infty_t(\sqrt{\mathbb{L}})}\big)
\|\omega\|_{L^\infty_t(\sqrt{\mathbb{L}})}.
\end{split}
\end{equation}
Finally,  we deal with the third  parentheses of $I_{3}$.  Since $(b\cdot\nabla) b=\frac{b_{\theta}^{2}}{r}e_r$ in the cylindrical coordinates,  we obtain that for $s<2$,
\begin{equation}\label{eq.I3}
\begin{split}
I_3&\leq Cp\int_0^t\Big\| \frac{b_{\theta}^{2}}{r}(\tau)\Big\|_{L^p}\,\mathrm{d}\tau\\
&\leq Cp\int_0^t\Big(\|b_{\theta}(\tau)\|^{2}_{L^{2p}}+\Big\| \frac{b_{\theta}}{r}(\tau)\Big\|^{2}_{L^{2p}}\Big)\,\mathrm{d}\tau
\end{split}
\end{equation}
Putting together \eqref{eq.I1}, \eqref{eq.I2} and \eqref{eq.I3} yields that for $s<2$,
\begin{align*}
\sum_{q\geq0}2^{qs}\|u_{q}\|_{L^1_t(L^p)}
&\leq  C(\|u_0\|_{L^2}+\|\omega_0\|_{\sqrt{\mathbb{L}}})+Ctp^{\frac{3}{2}}\big(\|u_0\|_{L^2}+
\|\omega\|_{L^\infty_t(\sqrt{\mathbb{L}})}\big)
\|\omega\|_{L^\infty_t(\sqrt{\mathbb{L}})}\\
&+Cp\int_0^t\Big(\|b_{\theta}(\tau)\|^{2}_{L^{2p}}+\Big\| \frac{b_{\theta}}{r}(\tau)\Big\|^{2}_{L^{2p}}\Big)\,\mathrm{d}\tau.
\end{align*}
Therefore, multiplying this inequality by $p^{-\frac32}$ and using  the previous estimates in Proposition \ref{b}, Proposition \ref{basic2} and Proposition \ref{pro1}, we have
\begin{equation*}
\begin{split}
\displaystyle\sup_{2\leq p<\infty}\int_0^t\sum_{q\geq0}2^{qs}\frac{\|u_{q}(\tau)\|_{L^p}}{p^{\frac32}}\,\mathrm{d}\tau
\leq & C,
\end{split}
\end{equation*}
where the positive constant $C$ depends only on $t$ and the initial data.

Now we are ready to prove \eqref{Loglip}.
With the help of the Sobolev inequality, for each $2\leq p<\infty$,  we have
\begin{equation}\label{vertical-estimate}
\norm{\partial_{z}u}_{L^p}\leq C\big\|\Lambda_{v}^{\frac34}u\big\|_{L^p}^{\frac34}
\big\|\Lambda_{v}^{\frac74}u\big\|_{L^p}^{\frac14}.
\end{equation}
On the one hand, by using the  interpolation inequality and Lemma \ref{prop-fra}, we get
\begin{equation*}
\big\|\Lambda_{v}^{\frac34}u\big\|_{L^p}\leq
C\big(\big\|\Lambda^{\frac{3}{4}}_vu\big\|_{L^2}+ \|\Lambda^{\frac{3}{4}}_vu\big\|_{L^\infty}\big)\leq
C\big(\|u\|_{L^2}+\|\omega\|_{\sqrt{\mathbb{L}}}\big).
\end{equation*}
On the other hand, by the Bernstein inequality, we have
\begin{equation*}
\big\|\Lambda_{v}^{\frac74}u\big\|_{L^p}\leq
\sum_{q\geq-1} \|\Delta_q^v\Lambda_{v}^{\frac74}u\|_{L^p}\leq
C\big(\|u\|_{L^2}+\sum_{q\geq0} 2^{\frac74q}{\|\Delta_q^vu\|_{L^p}}\big).
\end{equation*}
Combining these estimates with \eqref{vertical-estimate} leads to
$$\norm{\partial_{z}u}_{L^p}^{4}\leq C\big(\|u\|_{L^2}+\|\omega\|_{\sqrt{\mathbb{L}}}\big)^{3}\big(\|u\|_{L^2}+\sum_{q\geq0} 2^{\frac74q}{\|\Delta_q^vu\|_{L^p}}\big).$$
Hence, by using \eqref{eq-h-e} with $s=\frac{7}{4}$, we  deduce that
\begin{equation}
\int_{0}^{t}\frac{\norm{\partial_{z}u(\tau)}_{L^p}^{4}}{p^{\frac32}}\,\mathrm{d}\tau
\leq C\Big(1+\int_{0}^{t}\sum_{q\geq0}
2^{\frac74q}\frac{{\|\Delta_q^vu(\tau)\|_{L^p}}}{p^{\frac32}}\,\mathrm{d}\tau\Big)
<\infty.
\end{equation}
So finally, this combined with H\"{o}lder's inequality yields estimate \eqref{Loglip}.
\end{proof}

\subsection{Lipschitz bound of the velocity}
This subsection is devoted to showing the Lipschitz estimate for the velocity field via  losing estimates for the anisotropy system.
\begin{prop}\label{prop-losingestimate}
Let $(u_0,b_0)\in H^s\times H^s$ with $s>\frac52$. Assume that $(u_{0},b_{0})$ satisfies the conditions stated in Theorem \ref{the}. Let $(u, b)$ be the smooth  solution of  system \eqref{bs} satisfying \eqref{eq.axi-ass}.  Then, for any $T>0$, we have
\begin{equation}\label{eq.Lip}
\int_{0}^{T}\|\nabla u(t)\|_{L^{\infty}} \,\mathrm{d}t\leq C,
\end{equation}
where the positive constant $C$
depends on $T$ and the initial data.
\end{prop}

\begin{proof}
Since $(u_0,b_0)\in H^s\times H^s$ with $s>\frac52$, we know that the velocity vector field $u$ satisfies \eqref{l-5} and \eqref{l-4} by \eqref{basic0}, Proposition \ref{pro1} and Proposition \ref{prop-fine}. Since $u_0\in H^s$ with $s>\frac52$, we have $\|\frac{\omega_\theta}{r}(0)\|_{L^{3,1}}\leq C\|u_0\|_{H^s}$. In terms of Lemma \ref{con} and of Proposition \ref{basic}, we get that for $s>\frac52$,
\[\Big\|\frac{u_r}{r}(t)\Big\|_{L^\infty}\leq C\Big\|\frac{\omega_\theta}{r}(t)\Big\|_{L^{3,1}}\leq C\Big\|\frac{\omega_\theta}{r}(0)\Big\|_{L^{3,1}}\leq C\|u_0\|_{H^s}.\]
Therefore,  according to Proposition \ref{losingestimate-1}, for any $0<\sigma<1$, $0\leq\tau'<\tau<t\leq T$,  we have
\begin{equation}\label{losingpro}
\Big\|\frac{b_{\theta}}{r}\Big\|_{L^{\infty}_{T}(B_{\infty,\infty}^{\sigma_{\tau}})}\leq C(t)\Big\|\frac{b_{0}}{r}\Big\|_{B_{\infty,\infty}^{\sigma}},
\end{equation}
and
\begin{equation*}
\|b_{\theta}(\tau)\|_{B_{\infty,\infty}^{\sigma_{\tau}}}^{2}\leq C(t)\Big(\|b_{0}\|_{B_{\infty,\infty}^{\sigma}}^{2}+\int_{0}^{t}
\Big\|\Big(\frac{u_{r}}{r}b_{\theta}\Big)(\tau')\Big\|
_{B_{\infty,\infty}^{\sigma_{\tau'}}}^{2}\,\mathrm{d}\tau'\Big).
\end{equation*}
Here and in what follows, $C(t)$ is the smooth explicit function which may be different from line to line.

By using Lemma \ref{lem2.2}, we get
\begin{equation}\label{losingpro-0}\begin{split}
\Big\|\frac{u_{r}}{r}b_{\theta}\Big\|
_{B_{\infty,\infty}^{\sigma_{\tau'}}}\leq &C
\Big\|\frac{u_{r}}{r}\Big\|_{L^{\infty}}\|b_{\theta}\|
_{B_{\infty,\infty}^{\sigma_{\tau'}}}
+C\Big\|\frac{u_{r}}{r}\Big\|_{B_{\infty,\infty}^{\sigma_{\tau'}}}
\|b_{\theta}\|_{L^{\infty}}
\\ \leq &C\Big\|\frac{u_{r}}{r}\Big\|_{B_{\infty,\infty}^{\sigma_{\tau'}}}
\|b_{\theta}\|_{B_{\infty,\infty}^{\sigma_{\tau'}}}.
\end{split}
\end{equation}
By Bernstein inequality and Lemma \ref{point},  we deduce that for some $q\geq\frac{3}{1-\sigma_{\tau'}}$,
\begin{equation}\label{l--1}
\begin{split}
\Big\|\frac{u_{r}}{r}\Big\|_{B_{\infty,\infty}^{\sigma_{\tau'}}}=
&\sup_{j\geq-1}2^{j\sigma_{\tau'}}\Big\|\Delta_{j} \frac{u_{r}}{r}\Big\|_{L^{\infty}} \\ \leq&2^{-\sigma_{\tau'}}
\Big\|\Delta_{-1}\frac{u_{r}}{r}\Big\|_{L^\infty}+
\sup_{j\geq0}2^{j(\sigma_{\tau'}+\frac3q-1)}
\Big\|\Delta_{j}\partial_{z}\Big(\frac{u_{r}}{r}\Big)\Big\|_{L^q}\\ \leq&
C\Big(\Big\|\frac{u_{r}}{r}\Big\|_{L^\infty}
+\Big\|\frac{\omega_{\theta}}{r}\Big\|_{L^q}\Big).
\end{split}
\end{equation}
This together  with \eqref{add-w}, \eqref{basic0} and  Proposition \ref{basic} yields
\begin{equation*}
\Big\|\frac{u_{r}}{r}\Big\|_{B_{\infty,\infty}^{\sigma_{\tau'}}}
\leq C(t).
\end{equation*}
As a result,
\begin{equation*}
\|b_{\theta}(\tau)\|_{B_{\infty,\infty}^{\sigma_{\tau}}}^{2}\leq C(t)\Big(\|b_{0}\|_{B_{\infty,\infty}^{\sigma}}^{2}+\int_{0}^{t}
\|b_{\theta}(\tau')\|
_{B_{\infty,\infty}^{\sigma_{\tau'}}}^{2}\,\mathrm{d}\tau'\Big).
\end{equation*}
The Gronwall inequality implies that
\begin{equation}\label{b-Besov}
\sup_{0\leq\tau<\tau'}\|b_{\theta}(\tau)\|_{B_{\infty,\infty}^{\sigma_{\tau}}}\leq C(t)\|b_{0}\|_{B_{\infty,\infty}^{\sigma}}.
\end{equation}

Note that $\omega_{\theta}$  solves system \eqref{losingestimate} with $f=\frac{\omega_{\theta}}{r}u_{r}$, $g=-\frac{b_{\theta}^{2}}{r}$.  By using Proposition \ref{losingestimate--1}, we obtain that for the arbitrary $p\in[2,\infty),$
\begin{equation}\label{losingpro3}
\begin{split}
\|\omega_{\theta}(t)\|_{B_{p,\infty}^{\sigma_{t}}}^{2}\leq &
C(t)\Big(\|\omega_{0}\|_{B_{p,\infty}^{\sigma}}^{2}+
\int_{0}^{t}\Big\|\Big(\frac{u_{r}}{r}\omega_{\theta}\Big)(\tau)\Big\|
_{B_{p,\infty}^{\sigma_{\tau}}}^{2}\,\mathrm{d}\tau+
\int_{0}^{t}\Big\|\Big(\frac{b_{\theta}}{r}b_{\theta}\Big)(\tau)\Big\|
_{B_{p,\infty}^{\sigma_{\tau}}}^{2}\mathrm{d}\tau\Big).
\end{split}
\end{equation}
Now we need to estimate  two terms appearing on the right hand side of inequality \eqref{losingpro3}.
Since $0\leq\sigma_\tau<1$, by virtue of Lemma~\ref{lem2.2}, one has that
\begin{equation*}
\begin{split}
\Big\|\frac{u_{r}}{r}\omega_{\theta}\Big\|_{B_{p,\infty}^{\sigma_{\tau}}}
\leq& C\Big\|\frac{u_{r}}{r}\Big\|_{L^{\infty}}\|\omega_{\theta}\|
_{B_{p,\infty}^{\sigma_{\tau}}}
+C\|\omega_{\theta}\|_{L^{p}}
\Big\|\frac{u_{r}}{r}\Big\|_{B_{\infty,\infty}^{\sigma_{\tau}}}\\ \leq &C\Big\|\frac{u_{r}}{r}\Big\|_{B_{\infty,\infty}^{\sigma_{\tau}}}\|\omega_{\theta}\|
_{B_{p,\infty}^{\sigma_{\tau}}}\\ \leq &C(t)\|\omega_{\theta}\|
_{B_{p,\infty}^{\sigma_{\tau}}}.
\end{split}
\end{equation*}
In the last line of the inequality above, we have argued similarly as the proof of  \eqref{l--1} to get
$$\Big\|\frac{u_{r}}{r}\Big\|_{B_{\infty,\infty}^{\sigma_{\tau}}} \leq C(t).$$
By using Lemma \ref{lem2.2} again and \eqref{losingpro}, \eqref{b-Besov}, we can also bound $\big\|\big(\frac{b_{\theta}}{r}b_{\theta}\big)(\tau)\big\|
_{B_{p,\infty}^{\sigma_{\tau}}}$.
Thus, by the Gronwall inequality, we can get for the arbitrary $p\in[2,\infty),$
\begin{equation}\label{f2}
\sup_{0\leq t\leq T}\|\omega_{\theta}(t)\|_{B_{p,\infty}^{\sigma-\epsilon}}\leq\sup_{0\leq t\leq T}\|\omega_{\theta}(t)\|_{B_{p,\infty}^{\sigma_{t}}}\leq
C(t).
\end{equation}
Now we are in the position to prove that $$\int_{0}^{T}\|\nabla u(t)\|_{L^{\infty}}\,\mathrm{d}t<\infty,$$ which plays an important role in the proof of Theorem \ref{the}.
In fact, using the Bernstein inequality, we have that  by choosing $\sigma>\epsilon+\frac3p$  with sufficiently large $p$,
\begin{equation}\label{eq.delta}
\begin{split}
\|\nabla u\|_{L^{\infty}}\leq&\|\Delta_{-1}\nabla u\|_{L^{\infty}}+\sum_{j\geq0}\|\Delta_{j}\nabla u\|_{L^{\infty}}\\ \leq&
C(\|u\|_{L^{2}}+\sum_{j\geq0}2^{-j(\sigma-\epsilon-\frac3p)}2^{j(\sigma-\epsilon)}\|\Delta_{j}\omega\|_{L^\infty})\\ \leq&
C\Big(\|u_{0}\|_{L^{2}}+\|\omega\|_{B_{p,\infty}^{\sigma-\epsilon}}\Big).
\end{split}
\end{equation}
It follows by \eqref{f2} that $$\int_{0}^{T}\|\nabla u(t)\|_{L^{\infty}} \,\mathrm{d}t\leq C(T).$$
Since $p$ in \eqref{f2} and $\epsilon$ are arbitrary, estimate \eqref{eq.delta} holds for all $\sigma>0.$ On the other hand, we see that the initial datal data $(u_0,b_0)\in H^s\times H^s$ with $s>\frac52$, which guarantees that there exists a small enough  $\sigma>0$ such that $b_0\in B^\sigma_{\infty, \infty}$, $\frac{b_0}{r}\in B^\sigma_{\infty, \infty}$ and $\omega_\theta\in B^\sigma_{p,\infty}$ for  all $p\in[2,\infty]$. This completes  the proof.
\end{proof}

\subsection{High regularity for $(u,b)$}
In this subsection, we are going to derive the $H^{s}\times H^{s}$ $(s>\frac52)$ a priori estimates of $(u,b)$ associated to  system \eqref{bs} to gain the loss of regularity which occurs in Proposition~\ref{prop-losingestimate}.
\begin{prop}\label{high}
Let  $(u_{0},b_{0})\in H^{s}\times H^{s}$, $s>\frac{5}{2}$ satisfying the conditions stated in Theorem \ref{the}. Assume that $(u, b)$ be the smooth solution of system \eqref{bs}.  Then, for any $t\geq0$, there exists a constant $C>0$ depending  only on $t$ and the initial data such that
\begin{align*}
\Big\|\Big(\omega_{\theta}, b_{\theta},\nabla b_\theta,\frac{b_{\theta}}{r}\Big)(t)\Big\|^{2}_{H^{s-1}}+\int^{t}_{0}\norm{\partial_z\omega_{\theta}(\tau)}^{2}_{H^{s-1}}\,\mathrm{d}\tau\nonumber \leq C.
\end{align*}
\end{prop}
\begin{proof}
 Applying $\Delta_{q}$ to equality \eqref{b-equ} leads to
$$\partial_{t}\Delta_{q}\frac{b_{\theta}}{r}+
\big(S_{q+1}u\cdot\nabla\big)\Delta_{q}\frac{b_{\theta}}{r}
=R_{q}\Big(u,\frac{b_{\theta}}{r}\Big),$$ where $$R_{q}\Big(u,\frac{b_{\theta}}{r}\Big)= \big(S_{q+1}u\cdot\nabla\big)\Delta_q\frac{b_{\theta}}{r}-
\Delta_q\Big(\big(u\cdot\nabla\big) \frac{b_{\theta}}{r}\Big).$$
By taking the $L^{2}$-norm to this equation and using H\"{o}lder's inequality, we conclude that
\begin{equation}\label{add-0}
\begin{split}
&\frac12\frac{\mathrm{d}}{\mathrm{d}t}\Big\|\Delta_{q}\frac{b_{\theta}}{r}(t)\Big\|
^{2}_{L^{2}}\leq\Big\|R_{q}\Big(u,\frac{b_{\theta}}{r}\Big)\Big\|_{L^2}
\Big\|\Delta_{q}\frac{b_{\theta}}{r}\Big\|_{L^{2}}.
\end{split}
\end{equation}
By \eqref{eq-last-2}, we have
\begin{equation*}
\Big\|R_{q}\Big(u,\frac{b_{\theta}}{r}\Big)\Big\|_{L^2}\leq
C\|\nabla u\|_{L^\infty}\displaystyle{\sum_{q'\geq q
-4}}2^{q-q'}\Big\|\Delta_{q'}\frac{b_{\theta}}{r}\Big\|_{L^2}
+C\Big\|\frac{b_{\theta}}{r}\Big\|_{L^\infty}
\displaystyle{\sum_{|q'-q|\leq4}}\|\Delta_{q'}\nabla u\|_{L^2}.
\end{equation*}
Letting $\alpha=s-1.$ Plugging this commutator estimate into \eqref{add-0} and multiplying the resulting inequality by $2^{2q\alpha}$ and summing up over $q\geq-1$, we get
\begin{equation}\label{o}
\begin{split}
\frac{\mathrm{d}}{\mathrm{d}t}\Big\|\frac{b_{\theta}}{r}(t)\Big\|^{2}_{H^{\alpha}}&\leq C\|\nabla u\|_{L^{\infty}}\Big\|\frac{b_{\theta}}{r}\Big\|_{H^{\alpha}}^{2}+C\norm{\nabla u}_{H^{\alpha}}\Big\|\frac{b_{\theta}}{r}\Big\|_{L^{\infty}}\Big\|\frac{b_{\theta}}{r}\Big\|_{H^{\alpha}}.
\end{split}
\end{equation}

Now we turn to show the estimate of $b_{\theta}$.
Applying  operator $\Delta_{q}$ to the third equation of system~\eqref{MHDaxi}, we thus get
$$\partial_{t}\Delta_{q}b_{\theta}+
(S_{q+1}u\cdot\nabla)\Delta_{q}b_{\theta}
=R_{q}(u,b_{\theta})+\Delta_{q}\Big(\frac{u_{r}b_{\theta}}{r}\Big),$$
where $$R_{q}(u,b_{\theta})= (S_{q+1}u\cdot\nabla)\Delta_q b_{\theta}-\Delta_q\big((u\cdot\nabla )b_{\theta}\big).$$
In a similar way as to obtain \eqref{o}, we can get
\begin{equation*}
\begin{split}
\frac{\mathrm{d}}{\mathrm{d}t}\norm{b_{\theta}(t)}^{2}_{H^{\alpha}} &\leq C\Big(\norm{\nabla u}_{L^{\infty}}\norm{b_{\theta}}^{2}_{H^{\alpha}}+
\norm{\nabla u}_{H^{\alpha}}\norm{b_{\theta}}_{L^{\infty}}\norm{b_{\theta}}_{H^{\alpha}}+
\Big\|\frac{u_{r}b_{\theta}}{r}\Big\|_{H^{\alpha}}\norm{b_{\theta}}_{H^{\alpha}}\Big).
\end{split}
\end{equation*}
Moreover, by using  \eqref{eq-last-2}, we get
\begin{equation}\label{bb}
\begin{split}
\frac{\mathrm{d}}{\mathrm{d}t}\norm{b_{\theta}(t)}^{2}_{H^{\alpha}}\leq&
C\Big(\norm{\nabla u}_{L^{\infty}}\norm{b_{\theta}}^{2}_{H^{\alpha}}+
\norm{b_{\theta}}_{L^{\infty}}\norm{\nabla u}_{H^{\alpha}}\norm{b_{\theta}}_{H^{\alpha}}\\+&
\norm{b_{\theta}}_{L^{\infty}}\Big\|\frac{u_{r}}{r}\Big\|_{H^{\alpha}}\norm{b_{\theta}}_{H^{\alpha}}
+\Big\|\frac{u_{r}}{r}\Big\|_{L^{\infty}}\norm{b_{\theta}}^{2}_{H^{\alpha}}\Big).
\end{split}
\end{equation}

Finally, in order to get the  $H^{\alpha}$-estimate of $\omega_{\theta}$,
one may apply  operator $\Delta_{q}$ to the vorticity equation~\eqref{w} to obtain
$$\partial_{t}\Delta_{q}\omega_{\theta}+
(S_{q+1}u\cdot\nabla)\Delta_{q}\omega_{\theta}-\partial^{2}_{zz}\Delta_{q}\omega_{\theta}
=R_{q}(u,\omega_{\theta})+\Delta_{q}\Big(\frac{u_{r}\omega_{\theta}}{r}\Big)
-\Delta_{q}\Big(\frac{\partial_{z}b_{\theta}^{2}}{r}\Big),$$
where $$R_{q}(u,\omega_{\theta})= (S_{q+1}u\cdot\nabla)\Delta_q \omega_{\theta}-\Delta_q\big((u\cdot\nabla )\omega_{\theta}\big).$$
By taking the $L^{2}$-inner product with $\Delta_q\omega_{\theta}$ and using the incompressible condition, we  obtain
\begin{equation*}
\begin{split}
&\frac12\frac{\mathrm{d}}{\mathrm{d}t}\norm{\Delta_q\omega_{\theta}(t)}^{2}_{L^{2}}
+\norm{\partial_{z}\Delta_{q}\omega_{\theta}(t)}_{L^{2}}^{2}\\\leq&\norm{R_{q}(u,\omega_{\theta})}_{L^{2}}\norm{\Delta_{q}\omega_{\theta}}_{L^{2}}
+\Big\|\Delta_{q}\Big(\frac{u_{r}\omega_{\theta}}{r}\Big)\Big\|_{L^{2}}\norm{\Delta_{q}\omega_{\theta}}_{L^{2}}
+\Big\|\Delta_{q}\Big(\frac{b_{\theta}^{2}}{r}\Big)\Big\|_{L^{2}}\norm{\partial_{z}\Delta_{q}\omega_{\theta}}_{L^{2}}.
\end{split}
\end{equation*}
By using \eqref{eq-last-2} and Lemma \ref{lem2.2} again, multiplying both sides by $2^{2q\alpha}$ and summing up over $q\geq-1$, we have
\begin{equation}\label{vor}
\begin{split}
&\frac{\mathrm{d}}{\mathrm{d}t}\norm{\omega_{\theta}(t)}^{2}_{H^{\alpha}}+
\norm{\partial_{z}\omega_{\theta}(t)}^{2}_{H^{\alpha}}\\\leq& C\Big(\norm{\nabla u}_{L^{\infty}}\norm{\omega_{\theta}}^{2}_{H^{\alpha}}+
\norm{\omega_{\theta}}_{L^{\infty}}\norm{\nabla u}_{H^{\alpha}}\norm{\omega_{\theta}}_{H^{\alpha}}+
\Big\|\frac{u_{r}\omega_{\theta}}{r}\Big\|_{H^{\alpha}}\norm{\omega_{\theta}}_{H^{\alpha}}
+\Big\|\frac{b_{\theta}^{2}}{r}\Big\|^{2}_{H^{\alpha}}\Big)\\\leq&
C\Big(\norm{\nabla u}_{L^{\infty}}\norm{\omega_{\theta}}^{2}_{H^{\alpha}}+\norm{\omega_{\theta}}_{L^{\infty}}
\norm{\nabla u}_{H^{\alpha}}\norm{\omega_{\theta}}_{H^{\alpha}}+
\norm{\omega_{\theta}}_{L^{\infty}}\Big\|\frac{u_{r}}{r}\Big\|_{H^{\alpha}}\norm{\omega_{\theta}}_{H^{\alpha}}
\\&+\Big\|\frac{u_{r}}{r}\Big\|_{L^{\infty}}\norm{\omega_{\theta}}^{2}_{H^{\alpha}}+
\norm{b_{\theta}}^{2}_{L^{\infty}}\Big\|\frac{b_{\theta}}{r}\Big\|^{2}_{H^{\alpha}}+
\norm{b_{\theta}}^{2}_{H^{\alpha}}\Big\|\frac{b_{\theta}}{r}\Big\|^{2}_{L^{\infty}}\Big).
\end{split}
\end{equation}

By putting together these estimates \eqref{o}, \eqref{bb} and \eqref{vor} and using the Young inequality, we obtain
\begin{equation*}
\begin{split}
&\frac{\mathrm{d}}{\mathrm{d}t}\Big\|\Big(\omega_{\theta}, b_{\theta},\frac{b_{\theta}}{r}\Big)(t)\Big\|^{2}_{H^{\alpha}}+
\norm{\partial_{z}\omega_{\theta}(t)}^{2}_{H^{\alpha}}\\ \leq& C\Big(1+\|\nabla u\|_{L^{\infty}}+\norm{b_{\theta}}^{2}_{L^{\infty}}+
\Big\|\frac{b_{\theta}}{r}\Big\|^{2}_{L^{\infty}}\Big)\Big\|\Big(\omega_{\theta}, b_{\theta},\frac{b_{\theta}}{r}\Big)\Big\|^{2}_{H^{\alpha}}.
\end{split}
\end{equation*}
The facts that $\norm{\nabla u}_{H^\alpha}$ is equivalent to $\norm{\omega_{\theta}}_{H^\alpha}$ and $|\omega_{\theta}|\leq |\nabla u|$ are also used in the last inequality.

As a result, by the Gronwall lemma and the  estimates in Proposition \ref{b} and Proposition \ref{basic2}, we have
\begin{equation*}
\Big\|\Big(\omega_{\theta}, b_{\theta},\frac{b_{\theta}}{r}\Big)(t)\Big\|^{2}_{H^{\alpha}}+\int^{t}_{0}
\norm{\partial_{z}\omega_{\theta}(\tau)}^{2}_{H^{\alpha}}\,\mathrm{d}\tau
\leq C\Big\|\Big(\omega_{\theta}(0), b_{\theta}(0),\frac{b_{\theta}}{r}(0)\Big)\Big\|^{2}_{H^\alpha} e^{C\int_{0}^{t}\norm{\nabla u(\tau)}_{L^{\infty}}\,\mathrm{d}\tau}.
\end{equation*}
This combined with \mbox{Proposition \ref{prop-losingestimate}} provides
\begin{equation}\label{ff}
\Big\|\Big(\omega_{\theta}, b_{\theta},\frac{b_{\theta}}{r}\Big)(t)\Big\|^{2}_{H^{\alpha}}+\int^{t}_{0}
\norm{\partial_{z}\omega_{\theta}(\tau)}^{2}_{H^{\alpha}}\,\mathrm{d}\tau
\leq C.
\end{equation}
It remains for us to show estimate of $\|\nabla b_\theta\|_{H^\alpha}$. The classical commutator estimate helps us to conclude that
\begin{equation*}
\|\nabla b_\theta\|_{H^\alpha}^2\leq C\|\nabla b_\theta(0)\|_{H^\alpha}^2e^{\int_0^t\|\nabla u(\tau)\|_{H^s}\,\mathrm{d}\tau}.
\end{equation*}
Thanks to estimate \eqref{ff}, we finally obtain the desired result $\|b_\theta\|_{H^s}<\infty$ and then we completes the proof of the proposition.
\end{proof}

\section{Proof of Theorem \ref{the}}\label{Sec-4}
\setcounter{section}{4}\setcounter{equation}{0}
In this section, we restrict our attention to prove Theorem \ref{the}.
Firstly, we focus on the existence statement of Theorem \ref{the}. Let us begin with the following proposition which is about the local well-posedness for system \eqref{bs}.
\begin{prop}\label{Prop4.1}
Let $(u_{0},b_{0})\in H^{s}\times H^{s}$ with $s>\frac{5}{2}$. Then, there exists a maximal time $T>0$ depending only on $\norm{(u_{0}, b_{0})}_{H^s}$ such that system \eqref{bs} admits a unique local-in-time solution $(u,b)$ satisfying $u\in C([0,T);H^{s})$ and $b\in C([0,T);H^{s})$. Moreover, $\partial_z u\in L^{2}(0,T;H^{s})$.
\end{prop}

\begin{proof}
The result can be obtained by the Friedrichs method (see \cite{bcd} for more details): For $n\geq 1$, let $J_n$ be the spectral cut-off defined by
\begin{equation*}
\widehat{J_{n}f}(\xi)=1_{[0,n]}(|\xi|)\widehat{f}(\xi), \quad \xi \in\RR^3.
\end{equation*}
We consider the following system in the spaces $L^{2}_{n}:=\{f\in L^2(\RR^{3})|\,\text{supp}\, \widehat{f}\subset B(0,n)\}$:
\begin{equation}\label{appdexi}
\begin{cases}
\partial_tu+\mathcal{P}J_{n}\text{div}\,(\mathcal{P}J_nu\otimes \mathcal{P}J_nu)-\partial^{2}_{zz}J_nu=\mathcal{P}J_n\text{div}\,(\mathcal{P}J_nb\otimes \mathcal{P}J_nb),\\
\partial_tb +\mathcal{P}J_{n}\text{div}\,(\mathcal{P}J_nu\otimes \mathcal{P}J_nb)=\mathcal{P}J_{n}\text{div}\,(\mathcal{P}J_nb\otimes \mathcal{P}J_nu),\\
(u,b)|_{t=0}=J_{n}(u_0,b_0).
\end{cases}
\end{equation}
The Cauchy-Lipschitz theorem yields that there exists a unique maximal solution $(u_n,b_n)\in \mathcal{C}^{1}([0,T^{*}_{n});L^{2}_{n})$. Recall that $J^{2}_n=J_n, \mathcal{P}^{2}=\mathcal{P}$ and $J_n\mathcal{P}=\mathcal{P}J_n$, it is easy to check that $(\mathcal{P}u_n,\mathcal{P}b_n)$ and $(J_nu_n,J_nb_n)$ are also solutions.  By the uniqueness,  $\mathcal{P}u_n=u_n(\text{i.e. div} u_n=0)$, $J_nu_n=u_n$, $\mathcal{P}b_n=b_n(\text{i.e. div} b_n=0)$ and $J_nb_n=b_n$. Therefore, system \eqref{appdexi} can be simplified as
\begin{equation}\label{approx}
\begin{cases}
\partial_tu_{n}+\mathcal{P}J_{n}\text{div}\,(u_{n}\otimes u_{n})-\partial^{2}_{zz}u_{n}=\mathcal{P}J_n(b_{n}\otimes b_{n}),\\
\partial_tb_{n} +\mathcal{P}J_n\text{div}\,(u_{n} \otimes b_{n})=\mathcal{P}J_n\text{div}\,(b_{n} \otimes u_{n}),\\
\text{div}\,u_{n}=\text{div}\,b_{n}=0,\\
(u_{n},b_{n})|_{t=0}=J_{n}(u_0,b_0).
\end{cases}
\end{equation}
Since the operators $J_n$ and $\mathcal{P}J_n$ are the orthogonal projectors for the $L^2$-inner product. The classical commutator estimate enables us to conclude that the approximate solution $(u_n, b_n)$ of system \eqref{approx} satisfies
\begin{equation*}
\frac{\mathrm{d}}{\mathrm{d}t}\left\|\big(u_{n}, b_{n}\big)(t)\right\|^{2}_{H^{s}}+2
\norm{\partial_{z}u_{n}(\tau)}^{2}_{H^{s}}\mathrm{d}\tau\leq C\big(\|\nabla u_{n}\|_{L^\infty}+\|\nabla b_{n}\|_{L^\infty}\big)\left\|\big(u_{n}, b_{n}\big)(t)\right\|^{2}_{H^{s}}.
\end{equation*}
Since $s>\frac52$, the space $H^s(\mathbb{R}^3)$ continuously embeds in $W^{1,\infty}(\mathbb{R}^3)$,  it follows that
\begin{equation*}
\frac{\mathrm{d}}{\mathrm{d}t}X_n(t)\leq CX_n^2(t)
\end{equation*}
where  $ X_{n}(t):=\left\|\big(u_{n}, b_{n}\big)(t)\right\|_{H^{s}}.$

Then,  we get that for all $n$,
\begin{equation*}
\sup_{t\in[0,T]}\left\|\big(u_{n}, b_{n}\big)(t)\right\|_{H^{s}}\leq\frac{\left\|\big(u_0, b_0\big)\right\|_{H^{s}}}{1-CT\left\|\big(u_0, b_0\big)\right\|_{H^{s}}},
\end{equation*}
which implies that
\begin{equation*}
u_n\in L^{\infty}([0,T);H^{s}), ~b_n\in L^{\infty}([0,T);H^{s})~\text{and}~ \partial_{z}u_n\in L^{2}([0,T);H^{s})~ \text{are uniformly bounded},
\end{equation*}
provided that $T<(C\left\|\big(u_0, b_0\big)\right\|_{H^{s}})^{-1}$.  Hence, there exsits a couple $(u,b)$  such that $(u^n,b^n)\rightharpoonup(u,b)$ in $ L^{\infty}([0,T);H^{s})\times  L^{\infty}([0,T);H^{s}) $. According to  Fatou's Lemma, we have $(u,b)$ in $L^{\infty}([0,T);H^{s}\times H^s)$ and $\partial_{z} u\in L^{2}([0,T);H^{s})$.
By virtue of \mbox{equations \eqref{approx}} and uniform estimates of $(u_{n},b_{n})$,  it is easy to check that $\partial_{t}u_n\in L^{2}([0,T);H^{s-1})$
and $\partial_{t}b_n\in L^{\infty}([0,T);H^{s-2})$. Besides, we know that
$H^{s}\hookrightarrow H^{s-1}$ and $H^{s}\hookrightarrow H^{s-2}$ are locally compact. Therefore, by the classical Aubin-Lions argument and Cantor's diagonal process, we conclude that there exists a subsequence which we also denote $(u^n,b^n)$ such that $(u^n,b^n)\rightarrow(u,b)$ in $L^{2}([0,T);H^{s'})\times L^{2}([0,T);H^{s'})$ for all $s'<s.$ This strong convergence  enables us to derive  that the limit   $(u,b)$   is a distributional solution of problem \eqref{bs} on interval $[0,T).$ Since $(u,b)$ belongs to  $L^{\infty}([0,T);H^{s}\times H^s)$, the limit $u$ is a smooth local-in-time solution of problem \eqref{bs}. The time continuity follows from the fact that $u$ and $b$ satisfy transport equations with the velocity lying in Lipschitz field and  the source
 term belonging to $L^{2}([0,T);H^{s})$.

 Next, we show the uniqueness of solutions to problem \eqref{bs}. Suppose that $(u_1,p_1,b_1)$ and $(u_2,p_2,b_2)$ are  two solutions of system \eqref{bs} with the same initial data.
Letting the difference $(\delta u,\delta p,\delta b):=(u_1-u_2,p_1-p_2,b_1-b_2)$, we find that $(\delta u,\delta p,\delta b)$ solves
\begin{equation*}
\left\{\begin{array}{ll}
\partial_t\delta u+(u_1\cdot\nabla)\delta u-\partial_{zz}^2\delta u+\nabla\delta p=(b_1\cdot\nabla)\delta b+(\delta b\cdot\nabla)b_2-(\delta u\cdot\nabla)u_2,\quad(t,\mathrm{\mathbf{x}})\in\RR^+\times\RR^3,\\
\partial_t\delta b+(u_1\cdot\nabla)\delta b=(b_1\cdot\nabla)\delta u+(\delta b\cdot\nabla) u_2-(\delta u\cdot\nabla)b_2,\\
\text{div}\,\delta u=\text{div}\,\delta b=0,\\
(\delta u,\delta b)|_{t=0}=(0,0).
\end{array}\right.
\end{equation*}
Taking the standard $L^2$-estimate of $(\delta u,\delta b)$, we get
\begin{equation*}
\begin{split}
\frac12\frac{\mathrm{d}}{\mathrm{d}t}\big(\|\delta u(t)\|_{L^2}^2+\|\delta u(t)\|_{L^2}^2\big)+\|\partial_z\delta b(t)\|_{L^2}^2=&\int_{\RR^3}(\delta b\cdot\nabla)b_2\delta u\,\mathrm{d}\mathrm{\mathbf{x}}-\int_{\RR^3}(\delta u\cdot\nabla)u_2\delta u\,\mathrm{d}\mathrm{\mathbf{x}}\\
&+\int_{\RR^3}(\delta b\cdot\nabla) u_2\delta b\,\mathrm{d}\mathrm{\mathbf{x}}-\int_{\RR^3}(\delta u\cdot\nabla)b_2\delta b\,\mathrm{d}\mathrm{\mathbf{x}}.
\end{split}
\end{equation*}
Moreover, by the H\"older inequality, we have
\begin{equation*}
\frac{\mathrm{d}}{\mathrm{d}t}\big(\|\delta u(t)\|_{L^2}^2+\|\delta b(t)\|_{L^2}^2\big)\leq C\big(\|\nabla u_2\|_{L^\infty}+\|\nabla b_2\|_{L^\infty}\big)\big(\|\delta u(t)\|_{L^2}^2+\|\delta b(t)\|_{L^2}^2\big).
\end{equation*}
Since $\int_0^t\big(\|\nabla u_2(\tau)\|_{L^\infty}+\|\nabla b_2(\tau)\|_{L^\infty}\big)\,\mathrm{d}\tau<\infty$, the Gronwall inequality entails $(\delta u(t),\delta b(t))\equiv0$ on the whole interval $[0,T)$.
\end{proof}
Now, our main task is to show that this local-in-time solution can be extended to the global-in-time solution under the assumption that $u_{0}=u_{0}^{r}e_{r}+u_{0}^{z}e_{z}$ and $b_{0}=b_{0}^{\theta}e_{\theta}$. From Proposition \ref{Prop4.1}, we know that there exists  a smooth solution $(u, b)\in C([0,T);H^{s})\times C([0,T);H^{s})$ of system~\eqref{bs}. This together with the axisymmetric  assumption allows us to perform the same argument in Section \ref{Sec-3} to get $\int_{0}^{T}\|\nabla u(t)\|_{L^{\infty}}\,\mathrm{d}t<\infty.$ Then, by using  the BKM's criterion established in \cite{LZ}, we  get  that the local-in-time smooth solutions can  be extended to all the positive time. Thus, we complete the proof of Theorem \ref{the}.

\appendix
\section{Appendix}
\setcounter{section}{5}\setcounter{equation}{0}
In this section, we shall give two useful lemmas which have been used in sections above.

\begin{lem}\label{lem2.2}
Let $s>0$, $q\in[1,\infty]$. Then there exists a constant $C$ such that
the following inequality holds true
\begin{equation*}\|fg\|_{{
B}_{p,q}^s(\RR^n)}\leq C\big(\| f\|_{L^{p_1}(\RR^n)}\|g\|_{{ B}_{p_2,q}^s(\RR^n)}+\|
g\|_{L^{r_1}(\RR^n)}\|f\|_{{ B}_{r_2,q}^s(\RR^n)}\big),
\end{equation*}  where
$p_1,r_1\in[1,\infty]$ satisfy $\frac{1}{p}=\frac{1}{p_1}+\frac{1}{p_2}=\frac{1}{r_1}+\frac{1}{r_2}$.
\end{lem}
\begin{proof}
The proof of this lemma is standard, one can refer to \cite{CWZ2012} for the proof.
 \end{proof}
\begin{lem}[Commutator estimate]\label{commutator-est}
Let $1\leq p\leq\infty$ and $-1<\sigma<1$. Assume that $u$ is a divergence free vector-field over $\mathbb{R}^n$ and $\omega=\nabla\times u$.  Then,  there exists a positive constant $C$ such that for all $q\geq-1$, the term
$R_{q}(u,v):=S_{q+1}u\cdot\nabla\Delta_{q}v-\Delta_{q}(u\cdot\nabla v)$
satisfies  the following two estimates:
\begin{equation}\label{eq-last-2}
\|R_{q}(u,v)\|_{L^2(\mathbb{R}^n)}\leq
C\|\nabla u\|_{L^\infty(\RR^{n})}\displaystyle{\sum_{q'\geq q
-4}}2^{q-q'}\|\Delta_{q'}v\|_{L^2(\RR^{n})}+\|v\|_{L^\infty(\RR^{n})}
\displaystyle{\sum_{|q'-q|\leq5}}\|\Delta_{q'}\nabla u\|_{L^2(\RR^{n})}
\end{equation}
and
\begin{equation}\label{c-1}
\begin{split}
&\|R_q(u,v)\|_{L^p(\RR^{n})}\\
\leq& C \Big(\|S_{q+5}\nabla u\|_{L^\infty(\RR^{n})}\sum_{|q'-q|\leq5}\|\Delta_{q'}v\|_{L^p(\RR^{n})}+2^{-q\sigma}\sqrt{q+2}
\norm{\omega}_{\sqrt{\mathbb{L}}}\|v\|_{B_{p,\infty}^{\sigma}}\Big).
\end{split}
\end{equation}
\end{lem}
\begin{proof}
We omit the proof of \eqref{eq-last-2} because  its proof is standard and classical. One can refer to \cite{DP} for more details.

Let us begin to prove \eqref{c-1}.
We first decompose $R_q(u,v)$ as follows:
\begin{align*}
R_q(u,v)=&S_{q+1}u\cdot\nabla\Delta_qv-\Delta_q(S_{q+1}u\cdot\nabla v)-\Delta_q\big(({\rm I_d}-S_{q+1})u\cdot\nabla v\big)\nonumber\\
=&-[\Delta_q,S_{q+1}\bar u]\cdot\nabla v-[\Delta_q,S_{q+1}S_1u]\cdot\nabla v-\Delta_q\big(({\rm I_d}-S_{q+1})u\cdot\nabla v\big),
\end{align*}
where $\bar u=({\rm I_d}-S_1)u$.

By Bony's decomposition, we have
\begin{align*}
[\Delta_q,S_{q+1}\bar u]\cdot\nabla v=&[\Delta_q,T_{S_{q+1}\bar u_i}]\partial_i v+\Delta_q\big(T_{\partial_iv}S_{q+1}\bar u_i\big)+
\Delta_q\big(R(S_{q+1}\bar u_i,\partial_i v)\big)\nonumber\\
&-T_{\Delta_q\partial_iv}S_{q+1}\bar u_i-R(S_{q+1}\bar u_i,\Delta_q\partial_i v)\nonumber\\
:=&R^{1}_q(u,v)+R^{2}_q(u,v)+R^{3}_q(u,v)+R^{4}_q(u,v)+R^{5}_q(u,v),
\end{align*}
and
\begin{align*}
\Delta_q\big(({\rm I_d}-S_{q+1})u\cdot\nabla v\big)=&\Delta_q\big(T_{({\rm I_d}-S_{q+1})u_i}\partial_iv\big)+\Delta_q\big(T_{\partial_iv}({\rm I_d}-S_{q+1})u_i\big)+
\Delta_qR\big(({\rm I_d}-S_{q+1})u_i,\partial_iv\big)\nonumber\\
:=&R^{6}_q(u,v)+R^{7}_q(u,v)+R^{8}_q(u,v).
\end{align*}
From above, it is clear to find that the only term $[\Delta_q,S_{q+1}S_1u]\cdot\nabla v$ involves low frequency  of $u$.
First of all, we observe that
\begin{align*}
&[S_{q'-1}S_{q+1}\bar u_i,\Delta_q]\partial_i\Delta_{q'}v\nonumber\\
=&2^{qn}\int_{\mathbb{R}^n}\big(S_{q'-1}S_{q+1}\bar u_i(\mathrm{\mathbf{x}})-S_{q'-1}S_{q+1}\bar u_i(\mathrm{\mathbf{y}})\big)\varphi\big(2^{q}(\mathrm{\mathbf{x}}-\mathrm{\mathbf{y}})\big)\partial_i\Delta_{q'}v(\mathrm{\mathbf{y}})\,\mathrm{d}\mathrm{\mathbf{y}}\nonumber\\
=&-2^{qn}\int_{\mathbb{R}^n}\int_0^1\partial_kS_{q'-1}S_{q+1}\bar u_i\big(\tau \mathrm{\mathbf{x}}+(1-\tau)\mathrm{\mathbf{y}}\big)\,\mathrm{d}\tau(x_k-y_k)\varphi\big(2^{q}(\mathrm{\mathbf{x}}-\mathrm{\mathbf{y}})\big)\partial_i\Delta_{q'}v(\mathrm{\mathbf{y}})\,\mathrm{d}\mathrm{\mathbf{y}}\nonumber\\
=&-2^{q(n-1)}\int_{\mathbb{R}^n}\int_0^1\partial_kS_{q'-1}S_{q+1}\bar u_i\big(\tau \mathrm{\mathbf{x}}+(1-\tau)\mathrm{\mathbf{y}}\big)\,\mathrm{d}\tau 2^{q}(x_k-y_k)\varphi\big(2^{q}(\mathrm{\mathbf{x}}-\mathrm{\mathbf{y}})\big)\partial_i\Delta_{q'}v(\mathrm{\mathbf{y}})\,\mathrm{d}\mathrm{\mathbf{y}},
\end{align*}
where used the relation $\Delta_{q}f(\mathrm{\mathbf{x}})=2^{qn}\int_{\mathbb{R}^n}\varphi\big(2^{q}(\mathrm{\mathbf{x}}-\mathrm{\mathbf{y}})\big)f(\mathrm{\mathbf{y}})\,\mathrm{d}\mathrm{\mathbf{y}}$.

Therefore, we immediately get that
\begin{align*}
\|R^{1}_q(u,v)\|_{L^p(\RR^{n})}\leq&C \sum_{|q'-q|\leq5}2^{-q}\|\partial_k S_{q'-1}\bar u_i\|_{L^\infty(\RR^{n})}
\|\partial_i\Delta_{q'}v\|_{L^p(\RR^{n})}\int_{\mathbb{R}^n}|\mathrm{\mathbf{x}}\varphi(\mathrm{\mathbf{x}})|\,\mathrm{d}\mathrm{\mathbf{x}}\nonumber\\
\leq&C \sum_{|q'-q|\leq5}\|\partial_k S_{q'-1}u_i\|_{L^\infty(\RR^{n})}\|\Delta_{q'}v\|_{L^p(\RR^{n})}\nonumber\\
\leq& C\|S_{q+5}\nabla u\|_{L^\infty(\RR^{n})}\sum_{|q'-q|\leq5}\|\Delta_{q'}v\|_{L^p(\RR^{n})}.
\end{align*}
In a similar fashion as to prove  $R^1_q(u,v)$, we can bounded $[\Delta_q,S_{q+1}S_1u]\cdot\nabla v$ as follows:
\begin{align*}
\big\|[\Delta_q,S_{q+1}S_1u]\cdot\nabla v\big\|_{L^p(\RR^{n})}
\leq& C\|S_{q+5}\nabla u\|_{L^\infty(\RR^{n})}\sum_{|q'-q|\leq5}\|\Delta_{q'}v\|_{L^p(\RR^{n})}.
\end{align*}
For the second term $R_q^2(u,v)$, the H\"older inequality yields
\begin{align*}
\big\|R_q^2(u,v)\big\|_{L^p(\RR^{n})}\leq&C\sum_{|q'-q|\leq5}\|\Delta_{q'}
\bar u_i\|_{L^\infty(\RR^{n})}\big\|S_{q'-1}\partial_iv\big\|_{L^p(\RR^{n})}\nonumber\\
\leq&C\sum_{|q'-q|\leq5}2^{q-q'}\|\Delta_{q'}\nabla u_i\|_{L^\infty(\RR^{n})}\sum_{-1\leq k\leq q'-2}2^{k-q}\big\|\Delta_{k}v\big\|_{L^p(\RR^{n})}\\
\leq&C\sqrt{q+2} \|\omega\|_{\sqrt{\mathbb{L}}(\RR^{n})}
\sum_{-1\leq k\leq q+2}2^{k-q}\big\|\Delta_{k}v\big\|_{L^p(\RR^{n})}.\nonumber\\
\end{align*}
Similarly, we can conclude that
\begin{equation*}
\|R_q^4(u,v)\|_{L^p(\RR^{n})}\leq C\sqrt{q+2} \|\omega\|_{\sqrt{\mathbb{L}}(\RR^{n})}\sum_{-1\leq k\leq q+2}2^{k-q}\big\|\Delta_{k}v\big\|_{L^p(\RR^{n})}.
\end{equation*}
The reminder term $R_q^3(u,v)$ can be bounded by
\begin{align*}
\|\partial_i\Delta_q\big(R(S_{q+1}\bar u_i, v)\big)\|_{L^p(\RR^{n})}
\leq& C\sum_{q'\geq q-3}2^{q}\|\Delta_{q'}v\|_{L^p(\RR^{n})}\|\tilde\Delta_{q'}S_{q+1} \bar u_i\|_{L^\infty(\RR^{n})}\nonumber\\
\leq& C\sum_{q'\geq q-3}2^{q-q'}\|\Delta_{q'}v\|_{L^p(\RR^{n})}\|\tilde\Delta_{q'} \nabla \bar u_i\|_{L^\infty(\RR^{n})}\nonumber\\
\leq&C\sqrt{q+2} \|\omega\|_{\sqrt{\mathbb{L}}(\RR^{n})}\sum_{q'\geq q-3}\sqrt{1+(q'-q)}2^{-(q'-q)}\|\Delta_{q'}v\|_{L^p(\RR^{n})},
\end{align*}
where we have used the fact
\begin{equation*}
\|\tilde\Delta_{q'}\nabla u_i\|_{L^\infty(\RR^{n})}\leq\|\tilde\Delta_{q'} \omega\|_{L^\infty(\RR^{n})}\leq \sqrt{q'+2} \|\omega\|_{\sqrt{\mathbb{L}}(\RR^{n})}.
\end{equation*}
Similarly, we can conclude that
\begin{equation*}
\|R_q^5(u,v)\|_{L^p(\RR^{n})}
\leq C\sqrt{q+2} \|\omega\|_{\sqrt{\mathbb{L}}(\RR^{n})}\sum_{q'\geq q-3}\sqrt{1+(q'-q)}2^{-(q'-q)}\|\Delta_{q'}v\|_{L^p(\RR^{n})}.
\end{equation*}
It remains  for us to bound the last three terms $R_q^6(u,v),\,R_q^7(u,v)$ and $R_q^8(u,v)$. Thanks to the property of support and the H\"older inequality, one has
\begin{align*}
\|R_q^6(u,v)\|_{L^p(\RR^{n})}
\leq &C\sum_{|q'-q|\leq5}\|S_{q'-1}({\rm I_d}-S_{q+1})u_i\|_{L^\infty(\RR^{n})}\|\Delta_{q'}\partial_iv\|_{L^p(\RR^{n})}\nonumber\\
\leq&C\sum_{|q'-q|\leq5}2^{-q'}\|S_{q'-1}({\rm I_d}-S_{q+1})\nabla u_i\|_{L^\infty(\RR^{n})}\|\Delta_{q'}\partial_iv\|_{L^p(\RR^{n})}\nonumber\\
\leq&C\sum_{|q'-q|\leq5}\|S_{q'-1}\nabla u_i\|_{L^\infty(\RR^{n})}\|\Delta_{q'}v\|_{L^p(\RR^{n})}\nonumber\\
\leq&C\|S_{q+5}\nabla u\|_{L^\infty(\RR^{n})}\sum_{|q'-q|\leq5}\|\Delta_{q'}v\|_{L^p(\RR^{n})}.
\end{align*}
For the term $R_q^7(u,v)$, by the H\"older inequality, we obtain
\begin{align*}
\|R_q^7(u,v)\|_{L^p(\RR^{n})}
\leq &C\sum_{|q'-q|\leq5}\|S_{q'-1}\partial_iv\|_{L^p(\RR^{n})}\|\Delta_{q'}({\rm I_d}-S_{q+1})u_i\|_{L^\infty(\RR^{n})}\nonumber\\
\leq&C\sum_{|q'-q|\leq5}\,\sum_{-1\leq k\leq q'-2}2^{k-q'}\|\Delta_{k}v\|_{L^p(\RR^{n})}\|\Delta_{q'}\nabla u_i\|_{L^\infty(\RR^{n})}\nonumber\\
\leq&C\sqrt{q+2} \|\omega\|_{\sqrt{\mathbb{L}}(\RR^{n})}\sum_{-1\leq k\leq q+3}2^{k-q}\|\Delta_{k}v\|_{L^p(\RR^{n})}.
\end{align*}
As for the last term $R_q^8(u,v)$, by the H\"older inequality, we obtain
\begin{align*}
\|R_q^8(u,v)\|_{L^p(\RR^{n})}
\leq &C\big\|\partial_i\Delta_qR\big(({\rm I_d}-S_{q+1})u_i,v\big)\big\|_{L^p(\RR^{n})}\nonumber\\
\leq&C\sum_{q'\geq q-3}2^{q}\|\Delta_{q'}v\|_{L^p(\RR^{n})}\|\tilde\Delta_{q'}({\rm I_d}-S_{q+1})u_i\|_{L^\infty(\RR^{n})}\nonumber\\
\leq&C\sum_{q'\geq q-3}2^{q-q'}\|\Delta_{q'}v\|_{L^p(\RR^{n})}\|\tilde{\dot{\Delta}}_{q'}\nabla u_i\|_{L^\infty(\RR^{n})}\nonumber\\
\leq&C\sqrt{q+2} \|\omega\|_{\sqrt{\mathbb{L}}(\RR^{n})}\sum_{q'\geq q-3}\sqrt{1+(q'-q)}2^{-(q'-q)}\|\Delta_qv\|_{L^p(\RR^{n})}.
\end{align*}
Collecting  these estimates yields the desired result \eqref{c-1}.
\end{proof}

{\bf Acknowledgements.}
Jiu is partially supported  by National Natural Sciences Foundation of China (No.
11171229, No.11231006) and Project of Beijing Chang
Cheng Xue Zhe.

 \end{document}